\documentclass{amsart} 
\usepackage{amscd}
\usepackage{amsmath,empheq}
\usepackage{amsfonts}
\usepackage{amssymb}
\usepackage{mathrsfs}
\usepackage[all]{xy} 
\numberwithin{equation}{section}
\newtheorem{theorem}{Theorem}[section]
\newtheorem{ex}{Example}[section]
\newtheorem{lemma}[theorem]{Lemma}
\newtheorem{prop}{Proposition}[section]
\newtheorem{remark}{Remark}[section]

\newtheorem{claim}{Claim}[section]

\newenvironment{proof-sketch}{\noindent{\bf Sketch of Proof}\hspace*{1em}}{\qed\bigskip}
 
\everymath{\displaystyle}
\newcommand{\RR}{\mathbb R}
\newcommand{\NN}{\mathbb N}

\newcommand{\ZZ}{\mathbb Z}

\renewcommand{\leq}{\leqslant}

\renewcommand{\geq}{\geqslant}
\begin{document}
\title[Parametric nonlinear resonant Robin problems]{Parametric nonlinear resonant Robin problems}
\author[N.S. Papageorgiou]{N.S. Papageorgiou}
\address[N.S. Papageorgiou]{National Technical University, Department of Mathematics,
				Zografou Campus, Athens 15780, Greece  \&
Institute of Mathematics, Physics and Mechanics, 1000 Ljubljana, Slovenia
 }
\email{\tt npapg@math.ntua.gr}
\author[V.D. R\u{a}dulescu]{V.D. R\u{a}dulescu}
\address[V.D. R\u{a}dulescu]{Faculty of Applied Mathematics, AGH University of Science and Technology, 30-059 Krak\'ow, Poland
 \&  Institute of Mathematics, Physics and Mechanics, 1000 Ljubljana, Slovenia}
\email{\tt vicentiu.radulescu@imfm.si}
\author[D.D. Repov\v{s}]{D.D. Repov\v{s}}
\address[D.D. Repov\v{s}]{Faculty of Education and Faculty of Mathematics and Physics, University of Ljubljana \& Institute of Mathematics, Physics and Mechanics, 1000 Ljubljana, Slovenia}
\email{\tt dusan.repovs@guest.arnes.si}
\begin{abstract}
 We consider a nonlinear Robin problem driven by the $p$-Laplacian. In the reaction we have the competing effects of two nonlinearities. One term is parametric, strictly $(p-1)$-sublinear and the other one is $(p-1)$-linear and resonant at any nonprincipal variational eigenvalue. Using variational tools from the critical theory (critical groups), we show that for all large enough  values of  parameter $\lambda$ the problem has at least five nontrivial smooth solutions.
\end{abstract}
\maketitle
\section{Introduction}
Let $\Omega\subseteq\RR^N$ be a bounded domain with a $C^2$-boundary $\partial\Omega$. In this paper, we study the following parametric nonlinear Robin problem:
\begin{equation}
	\left\{
		\begin{array}{ll}
			-\Delta_p u(z) = \lambda g(z,u(z)) + f(z,u(z))\ \mbox{in}\ \Omega, \\
			\frac{\partial u}{\partial n_p} + \beta(z)|u|^{p-2}u=0\ \mbox{on}\ \partial\Omega,
		\end{array}
	\right\}
	\tag{$P_\lambda$}\label{eqp}
\end{equation}
where $\lambda$ is a positive parameter.

In this problem, $\Delta_p$ denotes the $p$-Laplace differential operator defined by
$$
\Delta_p u = {\rm div}\, (|Du|^{p-2}Du)\ \mbox{for all}\ u\in W^{1,p}(\Omega),\ 1<p<\infty.
$$

In the reaction (the right-hand side) of problem (\ref{eqp}),  $g(z,x)$ and $f(z,x)$ are Carath\'eodory functions (that is, for all $x\in\RR$, the mappings $ z\mapsto g(z,x)$ and $z\mapsto f(z,x)$ are measurable, while for almost all $z\in\Omega$, the mappings $ x\mapsto g(z,x)$ and $x\mapsto f(z,x)$ are continuous functions). These two nonlinearities exhibit different growth near $\pm\infty$ and $0$. More precisely, for almost all $z\in\Omega$, $g(z,\cdot)$ is $(p-1)$-sublinear both near $0$ and near $\pm\infty$, while $f(z,\cdot)$ is $(p-1)$-linear near $0$ and $\pm\infty$. In fact, we permit resonance at $\pm\infty$  with respect to any nonprincipal variational eigenvalue of $-\Delta_p$ with Robin boundary condition.

The coefficient $\beta(\cdot)$ that appears in the boundary condition is strictly positive. This is needed in order to be able to use strong comparison techniques, which in the case of the $p$-Laplace differential operator are difficult to have.

We denote by $\frac{\partial u}{\partial n_p}$ the conormal derivative of $u$, which is defined by extension of the map
\begin{equation*}
C^1(\overline\Omega)\ni u\mapsto \frac{\partial u}{\partial n_p} = |Du|^{p-2} (Du,n)_{\RR^N} = |Du|^{p-2}\frac{\partial u}{\partial n},
\end{equation*}
with $n(\cdot)$ being the outward unit normal on $\partial\Omega$.

Using variational tools from the critical point theory, together with suitable truncation and strong comparison techniques and Morse theory (critical groups), we show that for large enough  values of the positive parameter $\lambda$, problem (\ref{eqp}) admits at least five nontrivial smooth solutions.

Multiplicity results proving three solutions theorems for nonresonant Dirichlet $p$--Laplacian equations were established by Gasinski \& Papageorgiou \cite{6}, Guo \& Liu \cite{13}, and Jiu \& Su \cite{15}, Liu \cite{18}. Resonant $p$-Laplacian equations were investigated by Gasinski \& Papageorgiou \cite{7, 8},  Mugnai \& Papageorgiou \cite{20}, Papageorgiou \& R\u adulescu \cite{21} (Neumann problems), and Papageorgiou \& R\u adulescu \cite{22} (Robin problems). In all the above works, the resonance was with respect to the principal eigenvalue. Resonance with respect to higher variational eigenvalues was allowed in the recent works of Papageorgiou, R\u adulescu \& Repov\v{s} \cite{25, 26}, which dealt with nonparametric equations. None of the aforementioned works produces more than three solutions. Abstract methods closely related with the content of this paper have been developed in the recent monograph of Papageorgiou, R\u adulescu \& Repov\v{s} \cite{prr2019}.

\section{Mathematical background and hypotheses}

Let $X$ be a Banach space and $X^*$ be its topological dual. We denote by $\langle\cdot,\cdot\rangle$  the duality brackets for the pair $(X^*,X)$. Given $\varphi\in C^1(X,\RR)$, we say that $\varphi(\cdot)$ satisfies the ``Cerami condition" (the ``C-condition" for short), if the following property holds:
$$
\begin{array}{ll}
``\mbox{Every sequence}\ \{u_n\}_{n\geq1}\subseteq X\ \mbox{such that}
\{\varphi(u_n)\}_{n\geq1}\subseteq\RR\ \mbox{is bounded and}\\
 \lim_{n\rightarrow\infty} (1+||u_n||_X)\varphi'(u_n)=0\ \mbox{in}\ X^*, \,
\mbox{admits a strongly convergent subsequence}".
\end{array}
$$

This compactness-type condition on the functional $\varphi(\cdot)$ leads to a deformation theorem, which is the main analytical tool in deriving the minimax theory of the critical values of $\varphi$. One of the main results in that theory is the so-called ``mountain pass theorem", which we recall here.
\begin{theorem}\label{th1}
		If $\varphi\in C^1(X,\RR)$ satisfies the PS-condition, $u_0,\,u_1\in X$, $||u_1-u_0||>\rho>0$,
	$$\max\{\varphi(u_0),\varphi(u_1)\}<\inf\{\varphi(u):||u-u_0||=\rho\}=m_{\rho}$$
	and $c=\inf\limits_{\gamma\in\Gamma}\max\limits_{0\leq t\leq 1}\ \varphi(\gamma(t))$ with $\Gamma=\{\gamma\in C([0,1],X):\gamma(0)=u_0,\gamma(1)=u_1\}$, then $c\geq m_{\rho}$ and $c$ is a critical value of $\varphi$ (that is, we can find $\hat{u}\in X$ such that $\varphi'(\hat{u})=0$ and $\varphi(\hat{u})=c$).
\end{theorem}

The following spaces will play a central role in the analysis of problem (\ref{eqp}):
$$
W^{1,p}(\Omega),\ C^1(\overline\Omega)\ \mbox{and}\ L^p(\partial\Omega).
$$

We denote by $||\cdot||$ the norm of the Sobolev space $W^{1,p}(\Omega)$. We know that
$$
||u||=\left(||u||^p_p + ||Du||^p_p\right)^{\frac{1}{p}}\ \mbox{for all}\ u\in W^{1,p}(\Omega).
$$

The space $C^1(\overline\Omega)$ is an ordered Banach space with positive (order) cone
$$
C_+=\{u\in C^1(\overline\Omega): u(z)\geq0\ \mbox{for all}\ z\in\overline\Omega\}.
$$

This cone has a nonempty interior which contains the open set
$$
D_+ = \{u\in C_+: u(z)>0\ \mbox{for all}\ z\in\overline\Omega\}.
$$
In fact, $D_+$ is the interior of $C_+$ when the latter is furnished with the weaker $C(\overline\Omega)$-norm topology.

On $\partial\Omega$ we consider the $(N-1)$-dimensional Hausdorff (surface) measure $\sigma(\cdot)$. Using this measure on $\partial\Omega$, we can define in the usual way the boundary Lebesgue spaces $L^q(\partial\Omega)$, $ 1\leq q\leq\infty$. We know that there exists a unique continuous linear map $\gamma_0: W^{1,p}(\Omega)\rightarrow L^p(\partial\Omega)$, known as the ``trace map", such that
$$
\gamma_0(u)=u|_{\partial\Omega}\ \mbox{for all}\ u\in W^{1,p}(\Omega)\cap C(\overline\Omega).
$$

The trace map defines boundary values for all Sobolev functions. We know that $\gamma_0(\cdot)$ is a compact map into $L^q(\partial\Omega)$ for all $q\in\left[1,\frac{(N-1)p}{N-p}\right)$ when $p<N$, and into $L^q(\partial\Omega)$ for all $1\leq q<\infty$ when $N\leq p$. We have
$$
{\rm im}\,\gamma_0 = W^{\frac{1}{p'},p}(\partial\Omega)\ \mbox{and}\ \ker\gamma_0=W^{1,p}_0(\Omega).
$$

Recall that $p'$ denotes the conjugate exponent of $p$ (that is, $\frac{1}{p}+\frac{1}{p'}=1$). In what follows, for the sake of notational simplicity, we drop the use of trace map $\gamma_0$. All restrictions of Sobolev functions on $\partial\Omega$ are understood in the sense of traces.

Our hypotheses on the boundary coefficient $\beta(\cdot)$ are the following:

\smallskip
$H(\beta)$: $\beta\in C^{0,\alpha}(\partial\Omega)$ with $0<\alpha<1$ and $\beta(z)>0$ for all $z\in\partial\Omega$.

\smallskip
In the sequel, we denote by $\tau: W^{1,p}(\Omega)\rightarrow\RR$ the $C^1$-functional defined by
$$
\tau(u) = ||Du||^p_p + \int_{\partial\Omega}\beta(z)|u|^pd\sigma.
$$

By Proposition 2.4 of Gasinski \& Papageorgiou \cite{9}, we know that $\tau(\cdot)^\frac{1}{p}$ is an equivalent norm on $W^{1,p}(\Omega)$. So, there exist $c_1,c_2>0$ such that
\begin{equation}\label{eq1}
	c_1||u||^p\leq\tau(u)\leq c_2||u||^p\ \mbox{for all}\ u\in W^{1,p}(\Omega).
\end{equation}

Let $f_0:\Omega\times\RR\rightarrow\RR$ be a Carath\'eodory function such that
$$
|f_0(z,x)|\leq a_0(z)\left(1+|x|^{r-1}\right)\ \mbox{for almost all}\ z\in\Omega\ \mbox{and all}\ x\in\RR,
$$
with $a_0\in L^\infty(\Omega), 1<r\leq p^*$, where $p^*$ is the Sobolev critical exponent corresponding to $p$, hence
$$
p^*=\left\{
	\begin{array}{ll}
		\frac{Np}{N-p}\ &\mbox{if}\ p<N\\
		+\infty\ 		&\mbox{if}\ N\leq p.
	\end{array}
\right.
$$

We set $F_0(z,x)=\int^x_0 f_0(z,s)ds$ and consider the $C^1$-functional $\varphi_0: W^{1,p}(\Omega)\rightarrow\RR$ defined by
$$
\varphi_0(u) = \frac{1}{p}\tau(u) - \int_\Omega F_0(z,u)dz\ \mbox{for all}\ u\in W^{1,p}(\Omega).
$$

The next proposition is a special case of a more general result of Papageorgiou \& R\u adulescu \cite{23}. The proposition is essentially an outgrowth of the nonlinear regularity theory of Lieberman \cite{17}.

\begin{prop}\label{prop2}
	Assume that $u_0\in W^{1,p}(\Omega)$ is a local $C^1(\overline\Omega)$-minimizer of $\varphi_0$, that is, there exists $\rho_0>0$ such that
	$$
	\varphi_0(u_0)\leq\varphi_0(u_0+h)\ \mbox{for all}\ h\in C^1(\overline\Omega)\ \mbox{with}\ ||h||_{C^1(\overline\Omega)}\leq\rho_0,
	$$
	Then $u_0\in C^{1,\alpha}(\overline\Omega)$ for some $\alpha\in(0,1)$ and $u_0$ is also a local $W^{1,p}(\Omega)$-minimizer of $\varphi_0$, that is, there exists $\rho_1>0$ such that
	$$
	\varphi_0(u_0)\leq\varphi_0(u_0+h)\ \mbox{for all}\ h\in W^{1,p}(\Omega)\ \mbox{with}\ ||h||\leq \rho_1.
	$$
\end{prop}

It is well-known that in the nonlinear case ($p\neq2$), it is difficult to produce strong comparison results and more restrictive conditions are needed on the data of the problem. The next proposition is a special case of a more general result of Gasinski \& Papageorgiou \cite[Proposition 3.4]{9}.

\begin{prop}\label{prop3}
	If $h_1,h_2\in L^\infty(\Omega), h_1(z)\leq h_2(z)$ for almost all $z\in\Omega$, $h_1\not\equiv h_2$ and $u_1,u_2\in C^1(\overline\Omega)$ satisfy $u_1\leq u_2$ and
	\begin{eqnarray*}
		-\Delta_p u_1(z) = h_1(z)\ \mbox{for almost all}\ z\in\Omega,\ \frac{\partial u_1}{\partial n}|_{\partial\Omega}<0, \\
		-\Delta_pu_2(z) = h_2(z)\ \mbox{for almost all}\ z\in\Omega,\ \frac{\partial u_2}{\partial n}|_{\partial\Omega}<0,
	\end{eqnarray*}
	then $u_2-u_2\in {\rm int}\,C_+$.
\end{prop}

Let $A:W^{1,p}(\Omega)\rightarrow W^{1,p}(\Omega)^*$ be the nonlinear map defined by
$$
\langle A(u),h\rangle = \int_\Omega|Du|^{p-2} (Du,Dh)_{\RR^N}dz\ \mbox{for all}\ u,h\in W^{1,p}(\Omega).
$$

The next proposition is a special case of Problem 2.192 of Gasinski \& Papageorgiou \cite[p. 279]{11}.

\begin{prop}\label{prop4}
	The map $A(\cdot)$ is bounded (that is, it maps bounded sets to bounded sets), continuous, monotone (thus, maximal monotone, too) and of type $(S)_+$ (that is, if $u_n\xrightarrow{w}u$ in $W^{1,p}(\Omega)$ and $\limsup_{n\rightarrow\infty}\langle A(u_n),u_n-u\rangle\leq0$, then $u_n\rightarrow u$ in $W^{1,p}(\Omega)$).
\end{prop}

We will need some basic facts about the spectrum of the negative $p$-Laplacian with Robin boundary condition. So, we consider the following nonlinear eigenvalue problem:
\begin{equation}\label{eq2}
	\left\{
		\begin{array}{ll}
			-\Delta_pu(z) = \hat{\lambda}|u(z)|^{p-2}u(z)\ \mbox{in}\ \Omega, \\
			\frac{\partial u}{\partial n_p} + \beta(z)|u|^{p-2}u=0\ \mbox{on}\ \partial\Omega.
		\end{array}
	\right\}
\end{equation}

We say that $\hat{\lambda}\in\RR$ is an ``eigenvalue" of (\ref{eq2}), if the problem admits a nontrivial solution $\hat{u}\in W^{1,p}(\Omega)$, known as an eigenfunction corresponding to $\hat{\lambda}$. The nonlinear regularity theory of Lieberman \cite[Theorem 2]{17}, implies that $\hat{u}\in C^1(\overline\Omega)$. There is a smallest eigenvalue $\hat{\lambda}_1$ which has the following properties:
\begin{itemize}
	\item $\hat\lambda_1$ is isolated (that is, we can find $\varepsilon>0$ such that the open interval $(\hat\lambda_1,\hat\lambda_1+\varepsilon)$ contains no eigenvalues);
	\item $\hat\lambda_1$ is simple (that is, if $\hat{u}, \tilde{u}$ are eigenfunctions corresponding to $\hat\lambda_1$, then $\hat{u}=\xi\tilde{u}$ for some $\xi\in\RR\backslash\{0\}$);
\item we have
		\begin{equation}\label{eq3}
\hat\lambda_1=\inf\left\{\frac{\tau(u)}{||u||^p_p}:u\in W^{1,p}(\Omega), u\neq0\right\}>0\ \mbox{(see (\ref{eq1}))}.
		\end{equation}
\end{itemize}

The infimum in (\ref{eq3}) is realized on the corresponding one-dimensional eigenspace. From the above properties it follows that the elements of this eigenspace do not change sign and they, of course, belong in $C^1(\overline\Omega)$. Let $\hat{u}_1$ denote the positive, $L^p$-normalized (that is, $||\hat{u}_1||_p=1$) eigenfunction corresponding to $\hat{\lambda}_1$. We have $\hat{u}_1\in C_+\backslash\{0\}$ and in fact, by the nonlinear Hopf's boundary point theorem (see Gasinski \& Papageorgiou \cite[p. 738]{10}), we have $\hat{u}_1\in D_+$.

Let $\hat\sigma(p)$ denote the set of eigenvalues of (\ref{eq2}). It is easy to check that the set $\hat\sigma(p)\subseteq(0,+\infty)$ is closed. So, the second eigenvalue of (\ref{eq2}) is well-defined by
$$
\hat\lambda_2 = \min\{\hat\lambda\in\hat\sigma(p):\hat\lambda\neq\hat\lambda_1\}.
$$

The Ljusternik-Schnirelmann minimax scheme gives us in addition to $\hat\lambda_1$ and $\hat\lambda_2$, a whole strictly increasing sequence $\{\hat\lambda_k\}_{k\in\NN}$ of distinct eigenvalues of (\ref{eq2}) such that $\hat\lambda_k\rightarrow+\infty$. These are known as ``variational eigenvalues". Depending on the index used in the Ljusternik-Schnirelmann minimax scheme, we produce a corresponding sequence of variational eigenvalues. We know that these sequences coincide in the first two elements. However, we do not know if the variational eigenvalues are independent of the index used or they exhaust $\hat\sigma(p)$. This is the case if $p=2$ (linear eigenvalues problem). Here we consider the sequence of variational eigenvalues generated by the Fadell-Rabinowitz cohomological index (see \cite{5}). In this way we can use the results of Cingolani \& Degiovanni \cite{4} (see also Papageorgiou, R\u adulescu \& Repov\v{s} \cite[Proposition 12]{25}). Note that if $\hat\lambda\neq\hat\lambda_1$, then the eigenfunctions are sign-changing.

The following lemma is a simple consequence of the above properties of $\hat\lambda_1>0$ (see Papageorgiou, R\u adulescu \& Repov\v{s} \cite[Lemma 14]{25}).

\begin{lemma}\label{lem5}
	If $\vartheta\in L^\infty(\Omega), \vartheta(z)\leq\hat\lambda_1$ for almost all $z\in\Omega, \vartheta\not\equiv\hat\lambda_1$, then there exists $c_3>0$ such that
	\begin{equation*}
		c_3||u||^p\leq \tau(u) - \int_\Omega\vartheta(z)|u|^pdz\ \mbox{for all}\ u\in W^{1,p}(\Omega).
	\end{equation*}
\end{lemma}

Next, we recall some basic definitions and facts from the theory of critical groups. So, let $X$ be a Banach space and $\varphi\in C^1(X,\RR)$, $c\in\RR$. We introduce the following sets:
\begin{equation*}
	K_\varphi = \{u\in X:\varphi'(u)=0\},\ K^c_\varphi = \{u\in K_\varphi:\varphi(u)=c\},\ \varphi^c=\{u\in X:\varphi(u)\leq c\}.
\end{equation*}

Given a topological pair $(Y_1, Y_2)$ such that $Y_2\subseteq Y_1\subseteq X$, we denote by $H_k(Y_1, Y_2)$ ($k\in \NN_0$) the $k$th relative singular homology group with integer coefficients.
Recall that $H_k(Y_1, Y_2)=0$ for all $k\in -\NN$. Suppose that $u\in K^c_\varphi$ is isolated. The critical groups of $\varphi$ at $u$ are defined by
\begin{equation*}
	C_k(\varphi,u) = H_k(\varphi^c\cap U,\varphi^c\cap U\backslash\{u\})\ \mbox{for all}\ k\in \NN_0,
\end{equation*}
with $U$ a neighborhood of $u$ such that $K_\varphi\cap\varphi^c\cap U=\{u\}$. The excision property of singular homology implies that this definition is independent of the choice of the isolating neighborhood $U$.

Suppose that $\varphi$ satisfies the $C$-condition and $\inf\varphi(K_\varphi)>-\infty$. Then the critical groups of $\varphi$ at infinity are defined by
\begin{equation*}
	C_k(\varphi,\infty) = H_k(X,\varphi^c)\ \mbox{for all}\ k\in\NN_0,
\end{equation*}
with $c<\inf\varphi(K_\varphi)$. This definition is independent of the choice of the level $c<\inf\varphi(K_\varphi)$. Indeed, suppose that $c'<c<\inf\varphi(K_\varphi)$. Then the second deformation theorem (see, for example, Gasinski \& Papageorgiou \cite[p. 628]{10}) implies that $\varphi^{c'}$ is a strong deformation retract of $\varphi^c$. Therefore
\begin{eqnarray*}
	&&H_k(X,\varphi^c) = H_k(X,\varphi^{c'})\ \mbox{for all}\ k\in\NN_0\\
	&&\mbox{(see Motreanu, Motreanu \& Papageorgiou \cite[Corollary 6.15, p. 145]{19})}.
\end{eqnarray*}

Assume that $\varphi\in C^1(X,\RR)$ satisfies the C-condition and that $K_\varphi$ is finite. We introduce the following items:
\begin{eqnarray*}
	&&M(t,u) = \sum_{k\geq0}{\rm rank}\,C_k(\varphi,u)t^k\ \mbox{for all}\ t\in\RR,\ u\in K_{\varphi}, \\
	&&P(t,\infty) = \sum_{k\geq0}{\rm rank}\,C_k(\varphi,\infty)t^k\ \mbox{for all}\ t\in\RR.
\end{eqnarray*}

Then the Morse relation says that there exists $Q(t)=\sum_{k\geq0}\hat\beta_kt^k$ a formal series in $t\in\RR$ with nonnegative integer coefficients $\hat\beta_k$ such that
\begin{equation}\label{eq4}
	\sum_{u\in K_\varphi} M(t,u) = P(t,u) + (1+t)Q(t)\ \mbox{for all}\ t\in\RR.
\end{equation}

Now let us fix some basic notation which we will use throughout this work. So, for $x\in\RR$, we set $x^\pm=\max\{\pm x,0\}$. Then for $u\in W^{1,p}(\Omega)$ we define
\begin{equation*}
	u^\pm(\cdot)=u(\cdot)^\pm.
\end{equation*}

We know that
\begin{equation*}
	u^\pm\in W^{1,p}(\Omega),\ u=u^+-u^-,\ |u|=u^++u^-.
\end{equation*}

If $u,v\in W^{1,p}(\Omega)$ and $v\leq u$, then by $[v,u]$ we denote the ordered interval in $W^{1,p}(\Omega)$ defined by
\begin{equation*}
	[v,u]=\{y\in W^{1,p}(\Omega):v(z)\leq y(z)\leq u(z)\ \mbox{for almost all}\ z\in\Omega\}.
\end{equation*}

By ${\rm int}_{C^1(\overline\Omega)}[v,u]$, we denote the interior in the $C^1(\overline\Omega)$-norm topology of $[v,u]\cap C^1(\overline\Omega)$. We also define
$$[u)=\{y\in W^{1,p}(\Omega): u(z)\leq y(z)\ \mbox{for almost all}\ z\in\Omega\}.$$

For $u,v\in W^{1,p}(\Omega)$ with $v(z)\neq0$ for almost all $z\in\Omega$, we define
\begin{equation*}
	R(u,v)(z) = |Du(z)|^p - |Dv(z)|^{p-2} (Dv(z), D\left(\frac{u^p}{v^{p-1}}\right)(z))_{\RR^N},\ z\in\Omega.
\end{equation*}

From the nonlinear Picone's identity of Allegretto \& Huang \cite{2}, we have the following property.

\begin{prop}\label{prop6}
	If $u,v:\Omega\rightarrow\RR$ are differentiable functions with $u(z)\geq0$ and $ v(z)>0$ for all $z\in\Omega$, then $R(u,v)(z)\geq0$ for almost all $z\in\Omega$ and equality holds if and only if $u=\xi v$ with $\xi\geq0$.
\end{prop}

Finally, if $k,m\in\NN_0$, then by $\delta_{k,m}$ we denote the Kronecker symbol, that is,
\begin{equation*}
	\delta_{k,m}=\left\{
		\begin{array}{ll}
			1\ \mbox{if}\ k=m\\
			0\ \mbox{if}\ k\neq m.			
		\end{array}
	\right.
\end{equation*}

Next, we introduce our hypotheses on the two nonlinearities in the reaction of (\ref{eqp}).

\smallskip
$H(g):$ $g:\Omega\times\RR\rightarrow\RR$ is a Carath\'eodory function such that $g(z,0)=0$ for almost all $z\in\Omega$ and
\begin{itemize}
	\item [(i)] for every $\rho>0$, there exists $a_\rho\in L^\infty(\Omega)$ such that
		\begin{eqnarray*}
			|g(z,x)| \leq a_\rho(z)\ \mbox{for almost all}\ z\in\Omega\ \mbox{and all}\ |x|\leq\rho, \\
			0< g(z,x) x\ \mbox{for almost all}\ z\in\Omega\ \mbox{and all}\ x\in\RR\backslash\{0\};
		\end{eqnarray*}
	\item [(ii)] $\lim_{x\rightarrow0}\frac{g(z,x)}{|x|^{p-2}x}=0$ and there exists $1<q<p$ such that $\lim_{x\rightarrow\pm\infty}\frac{g(z,x)}{|x|^{q-2}x}=0$ uniformly for almost all $z\in\Omega$.
\end{itemize}

\smallskip
$H(f):$ $f:\Omega\times\RR\rightarrow\RR$ is a Carath\'eodory function such that $f(z,0)=0$ for almost all $z\in\Omega$ and
\begin{itemize}
	\item [(i)] for every $\rho>0$, there exists $\hat{a}_\rho\in L^\infty(\Omega)$ such that
		\begin{equation*}
			|f(z,x)|\leq\hat{a}_\rho(z)\ \mbox{for almost all}\ z\in\Omega\ \mbox{and all}\ |x|\leq\rho;
		\end{equation*}
	\item [(ii)] $\lim_{x\rightarrow\pm\infty}\frac{f(z,x)}{|x|^{p-2}x}=\hat\lambda_m$ uniformly for almost all $z\in\Omega$ for some $m\in\NN, m\geq2$ and if $F(z,x)=\int^x_0 f(z,s)ds$ then $\liminf_{x\rightarrow\pm\infty}\frac{p F(z,x) - f(z,x)x}{|x|^q}>0$ uniformly for almost all $z\in\Omega$;
	\item [(iii)] there exists $\vartheta\in L^\infty(\Omega)$ such that
		$$
		\begin{array}{ll}
			\vartheta(z)\leq\hat\lambda_1\ \mbox{for almost all}\ z\in\Omega, \vartheta\not\equiv\hat\lambda_1, \\
			\limsup_{x\rightarrow\pm\infty}\frac{f(z,x)}{|x|^{p-2}x}\leq \vartheta(z)\ \mbox{uniformly for almost all}\ z\in\Omega.
		\end{array}
		$$
\end{itemize}

$H_0:$ For almost all $z\in\Omega$ and every $\lambda>0$, the mapping $x\mapsto\lambda g(z,x) + f(z,x)$ is strictly increasing.

\begin{remark}\label{rem1}
	Hypothesis $H(g)(ii)$ implies that $g(z,\cdot)$ is strictly sublinear near $\pm\infty$ and $0$. On the other hand, hypothesis $H(f)(ii)$ implies that $f(z,\cdot)$ is $(p-1)$-linear near $\pm\infty$ and $0$. Note that hypotheses $H(g)(ii), H(f)(ii)$ imply that problem (\ref{eqp}) at $\pm\infty$ is resonant with respect to a nonprincipal variational eigenvalue of the Robin p-Laplacian. Clearly, the above hypotheses imply that
	\begin{equation}\label{eq5}
		|g(z,x)|,\ |f(z,x)|\leq c_4(1+|x|^{p-1})\ \mbox{for almost all}\ z\in\Omega\ \mbox{and all}\ x\in\RR
	\end{equation}
	with $c_4>0$. In the sequel, we shall denote $G(z,x)=\int^x_0 g(z,s)ds$.
\end{remark}

\begin{ex}\label{ex1}
	The following functions satisfy hypotheses $H(g), H(f)$. For the sake of simplicity, we drop the z-dependence.
	\begin{eqnarray*}
		&&g(x)=\left\{
			\begin{array}{ll}
				|x|^{r-2}x\ &\mbox{if}\ |x|\leq1 \\
				|x|^{s-2}x\ &\mbox{if}\ 1<|x|
			\end{array}
		\right.\quad 1<s<p<r; \\
		&&f(x)=\left\{
			\begin{array}{ll}
				\vartheta|x|^{p-2}x\ &\mbox{if}\ |x|\leq1 \\
				\hat\lambda_m|x|^{p-2}x + (\hat\lambda_m-\vartheta)|x|^{q-2}x\ &\mbox{if}\ 1<|x|
			\end{array}
		\right.,\ s<q<p,\ \vartheta<\hat\lambda_1.
	\end{eqnarray*}
\end{ex}

\section{Solutions of constant sign}

On account of hypotheses $H(g)(ii), H(f)(ii)$ and (\ref{eq5}), we see that given $\lambda>0$, $\epsilon>0$ and $r\in(p,p^*)$, we can find $c_5>0$ such that
\begin{equation}\label{eq6}
	[\lambda g(z,x) + f(z,x)]x \leq [\vartheta(z) + (1+\lambda)\varepsilon] |x|^p + c_5|x|^r\ \mbox{for almost all}\ z\in\Omega\ \mbox{and all}\ x\in\RR.
\end{equation}

This unilateral growth restriction on the reaction of problem (\ref{eqp}) leads to the following auxiliary parametric nonlinear Robin problem
\begin{equation}\label{eq7l}
	\left\{
		\begin{array}{ll}
			-\Delta_p u(z)=\left(\vartheta(z)+(1+\lambda)\varepsilon\right) |u(z)|^{p-2}u(z) + c_5|u(z)|^{r-2}u(z)\ \mbox{in}\ \Omega, \\
			\frac{\partial u}{\partial n_p} + \beta(z)|u|^{p-2}u=0\ \mbox{on}\ \partial\Omega.	
		\end{array}
	\right\}
\end{equation}

\begin{prop}\label{prop7}
	If hypothesis $H(\beta)$ holds and $\lambda>0$, then for every sufficiently small $\varepsilon>0$ problem (\ref{eq7l}) admits a positive solution
	\begin{equation*}
		\tilde{u}_\lambda\in D_+.
	\end{equation*}
	Moreover, since (\ref{eq7l}) is odd,  $\tilde{v}_\lambda=-\tilde{u}_\lambda\in -D_+$ is a negative solution of problem (\ref{eq7l}).
\end{prop}

\begin{proof}
	Let $\Psi^+_\lambda: W^{1,p}(\Omega)\rightarrow\RR$ be the $C^1$-functional defined by
	\begin{equation*}
		\Psi^+_\lambda(u)=\frac{1}{p}\tau(u) - \frac{1}{p}\int_\Omega[\vartheta(z) + (1+\lambda)\varepsilon](u^+)^pdz - \frac{c_5}{r}||u^+||^r_r\ \mbox{for all}\ u\in W^{1,p}(\Omega).
	\end{equation*}
	
	We have
	\begin{eqnarray*}
		\Psi^+_\lambda(u)\geq c_6||u^-||^p + \frac{1}{p}\left(\tau(u^+) - \int_\Omega\vartheta(z)(u^+)^pdz\right) - \frac{(1+\lambda)\varepsilon}{p}||u^+||^p - c_7||u||^r\\
		 \mbox{for some}\ c_6, c_7>0\ \mbox{(see (\ref{eq1}))} \\
		 \geq c_6 ||u^-||^p + \frac{1}{p}[c_8 - (1+\lambda)\varepsilon]||u^+||^p - c_7||u||^r\ \mbox{for all}\ u\in W^{1,p}(\Omega),\ \mbox{some}\ c_8>0.
	\end{eqnarray*}
	
	Choosing $\varepsilon\in \left(0, \frac{c_8}{1+\lambda}\right)$, we consider that
	\begin{equation}\label{eq8}
		\Psi^+_\lambda(u)\geq c_9||u||^p - c_7||u||^r\ \mbox{for all}\ u\in W^{1,p}(\Omega),\ \mbox{some}\ c_9>0.
	\end{equation}
	
	Since $r>p$, it follows from (\ref{eq8}) that
	\begin{equation*}
		u=0\ \mbox{is a local minimizer of}\ \Psi^+_\lambda.
	\end{equation*}
	
	Then we can find so small $\rho\in(0,1)$   that
	\begin{equation}\label{eq9}\begin{array}{ll}
		&\displaystyle \Psi^+_\lambda(0)=0<\inf\{\Psi^+_\lambda(u):||u||=\rho\}=m^+_\lambda 	
	 \nonumber
	\end{array}
\end{equation}
	(see Aizicovici, Papageorgiou \& Staicu \cite{1}, proof of Proposition 29).
	
	For $t>0$, we have
	\begin{eqnarray}\label{eq10}
		\Psi^+_\lambda(t\hat{u}_1) & =&\frac{t^p}{p}\tau(\hat{u}_1) -  \frac{t^p}{p}\int_\Omega[\vartheta(z) + (1+\lambda)\varepsilon]\hat{u}_1^pdz - \frac{t^r}{r}||\hat{u}_1||^r_r \nonumber \\
		& \leq & \frac{t^p}{p}\int_\Omega[\hat\lambda_1 - \vartheta(z)]\hat{u}_1^pdz - \frac{t^r}{r}||\hat{u}_1||^r_r \nonumber \\
		& \leq & c_{10}t^p - c_{11}t^r\ \mbox{for some}\ c_{10},\,c_{11}>0.
	\end{eqnarray}
	
	However, $r>p$. So, from (\ref{eq10}) we have
	\begin{equation}\label{eq11}
		\Psi^+_\lambda(t\hat{u}_1)\rightarrow-\infty\ \mbox{as}\ t\rightarrow+\infty.
	\end{equation}
	
	Let $k_\lambda(z,x)$ be the Carath\'eodory function defined by
	\begin{equation*}
		k_\lambda(z,x) = [\vartheta(z) + (1+\lambda)\varepsilon] |x|^{p-2}x + c_5|x|^{r-2}x.
	\end{equation*}
	
	We set
	\begin{equation*}
		K_\lambda(z,x) = \int^x_0 k_\lambda(z,s)ds = \frac{1}{p}[\vartheta(z)+(1+\lambda)\varepsilon]|x|^p + \frac{c_5}{r}|x|^r.
	\end{equation*}
	
	Recall that $p<r$ and let $q\in(p,r)$. For sufficiently large $M>0$ we have
	\begin{eqnarray}
		&& 0<qK_\lambda(z,x)\leq k_\lambda(z,x)x\ \mbox{for almost all}\ z\in\Omega\ \mbox{and all}\ |x|\geq M, \nonumber \\
		&\Rightarrow & k_\lambda(z,\cdot)\ \mbox{satisfies the Ambrosetti-Rabinowitz condition (see \cite[p. 341]{19})} \nonumber \\
		&\Rightarrow & \Psi^+_\lambda(\cdot)\ \mbox{satisfies the C-condition (see \cite[p. 343]{19})}. \label{eq12}
	\end{eqnarray}
	
	Then (\ref{eq9}), (\ref{eq11}), (\ref{eq12}) permit the use of Theorem \ref{th1} (the mountain pass theorem). So, we can find $\tilde{u}_\lambda\in W^{1,p}(\Omega)$ such that
	\begin{equation*}
		\tilde{u}_\lambda\in K_{\Psi^+_\lambda}\ \mbox{and}\ \Psi^+_\lambda(0)=0<m^+_\lambda\leq \Psi^+_\lambda(\tilde{u}_\lambda).
	\end{equation*}
	
	Evidently, $\tilde{u}_\lambda\neq0$ and we have
	\begin{eqnarray}
		&& \langle A(\tilde{u}_\lambda),h\rangle + \int_{\partial\Omega}\beta(z)|\tilde{u}_\lambda|^{p-2}\tilde{u}_\lambda hd\sigma  \nonumber \\
	 &= & 
		\int_\Omega\left\{[\vartheta(z) + (1+\lambda)\epsilon](\tilde{u}^+\lambda)^{p-1} + c_5(\tilde{u}^+_\lambda)^{r-1}\right\}hdz\ \mbox{for all}\ h\in W^{1,p}(\Omega). \label{eq13}
	\end{eqnarray}
	
	In (\ref{eq13}) we choose $h=-\tilde{u}^-_\lambda\in W^{1,p}(\Omega)$. Then
	\begin{eqnarray*}
		&& \tau(\tilde{u}^-_\lambda)=0, \\
		&\Rightarrow & \tilde{u}_\lambda\geq0,\ \tilde{u}_\lambda\neq0\ \mbox{(see (\ref{eq1}))}.
	\end{eqnarray*}
	
	Then by (\ref{eq13}) we have
	\begin{eqnarray}
		-\Delta_p\tilde{u}_\lambda(z) = [\vartheta(z) + (1+\lambda)\varepsilon]\tilde{u}_\lambda(z)^{p-1} + c_5\tilde{u}_\lambda(z)^{r-1}\ \mbox{for almost all}\ z\in\Omega, \nonumber \\
		\frac{\partial \tilde{u}_\lambda}{\partial n_p} + \beta(z)\tilde{u}^{p-1}_\lambda=0\ \mbox{on}\ \partial\Omega\ \mbox{(see Papageorgiou \& R\u adulescu \cite{22})}. \label{eq14}
	\end{eqnarray}
	
	By (\ref{eq14}) and Proposition 7 of Papageorgiou \& R\u adulescu \cite{23}, we have
	\begin{equation*}
		\tilde{u}_\lambda\in L^\infty(\Omega).
	\end{equation*}
	
	So, we can apply Theorem 2 of Lieberman \cite{17} and conclude that
	\begin{equation*}
		\tilde{u}_\lambda\in C_+\backslash\{0\}.
	\end{equation*}
	
	It follows from (\ref{eq14}) that
	\begin{eqnarray*}
		& \Delta_p\tilde{u}_\lambda(z)\leq 0\ \mbox{for almost all}\ z\in\Omega, \\
		\Rightarrow & \tilde{u}_\lambda\in D_+\ \mbox{(see Gasinski \& Papageorgiou \cite[p. 738]{10})}.
	\end{eqnarray*}
	
	Since problem (\ref{eq7l}) is odd, we can deduce that $\tilde{v}_\lambda=-\tilde{u}_\lambda\in -D_+$ is a negative solution of problem (\ref{eq7l}).
\end{proof}

Next, we produce a uniform lower bound $\hat{c}>0$ for the solutions $\tilde{u}_\lambda$ of (\ref{eq7l}) for $\lambda>0$. It follows that $-\hat{c}<0$ is an upper bound for the negative solutions $\tilde{v}_\lambda$.

\begin{prop}\label{prop8}
	If hypothesis $H(\beta)$ holds, then there exists $\hat{c}>0$ such that
	\begin{equation*}
		\hat{c}\leq\tilde{u}_\lambda(z)\ \mbox{and}\ \tilde{v}_\lambda(z)\leq -\hat{c}\ \mbox{for all}\ z\in\overline\Omega,\ \lambda>0.
	\end{equation*}
\end{prop}

\begin{proof}
	We consider the following nonlinear Robin problem
	\begin{equation}\label{eq15}
		\left\{
			\begin{array}{ll}
				-\Delta_p u(z) = c_5|u(z)|^{r-2}u(z)\ \mbox{in}\ \Omega, \\
				\frac{\partial u}{\partial n_p} + \beta(z)|u|^{p-2}u=0\ \mbox{on}\ \partial\Omega.
			\end{array}
		\right\}
	\end{equation}
	
	We first show that problem (\ref{eq15}) has a positive solution. So, let $\xi: W^{1,p}(\Omega)\rightarrow\RR$ be the $C^1$-functional defined by
	\begin{equation*}
		\xi(u) = \frac{1}{p}\tau(u) - \frac{c_5}{r}||u^+||^r_r\ \mbox{for all}\ u\in W^{1,p}(\Omega).
	\end{equation*}
	
	Using (\ref{eq1}) we have
	\begin{eqnarray*}
		& \xi(u)\geq c_{12}||u||^p - c_{13}||u||^r\ \mbox{for some}\ c_{12},c_{13}>0,\ \mbox{all}\ u\in W^{1,p}(\Omega), \\
		\Rightarrow & u=0\ \mbox{is an isolated local minimizer of}\ \xi(\cdot)\ \mbox{(recall that $r>p$)}.
	\end{eqnarray*}
	
	So, we can find $\rho\in(0,1)$ small such that
	\begin{equation}\label{eq16}
		\xi(0)=0<\inf\{\xi(u):||u||=\rho\}=m_\xi.
	\end{equation}
	
	Also, if $u\in D_+$, then
	\begin{equation}\label{eq17}
		\xi(tu)\rightarrow-\infty\ \mbox{as}\ t\rightarrow+\infty\ \mbox{(again use the fact that $r>p$).}
	\end{equation}
	
	Finally, since the reaction $f(x)=c_5(x^+)^{p-1}$ satisfies the Ambrosetti-Rabinowitz condition on $\RR_+=[0,+\infty)$, we can infer that
	\begin{equation}\label{eq18}
		\xi(\cdot)\ \mbox{satisfies the C-condition}.
	\end{equation}
	
	Then (\ref{eq16}), (\ref{eq17}), (\ref{eq18}) permit the use of Theorem \ref{th1} (the mountain pass theorem) and we obtain $\overline{u}\in W^{1,p}(\Omega)$ such that
	\begin{equation}\label{eq19}
		\overline{u}\in K_\xi\ \mbox{and}\ \xi(0)=0<m_\xi\leq\xi(\overline{u}).
	\end{equation}
	
	From (\ref{eq19}) we can infer that $\overline{u}\neq0$ and
	\begin{eqnarray*}
		&& \xi'(\overline{u})=0, \\
		&\Rightarrow & \langle A(\overline{u}),h\rangle + \int_{\partial\Omega}\beta(z)|\overline{u}|^{p-2}\overline{u}hd\sigma = c_5\int_\Omega (\overline{u}^+)^{p-1}hdz\ \mbox{for all}\ h\in W^{1,p}(\Omega).
	\end{eqnarray*}
	
	We choose $h=-\overline{u}^-\in W^{1,p}(\Omega)$. Then
	\begin{eqnarray*}
		& \tau(\overline{u}^-) = 0, \\
		\Rightarrow & \overline{u}\geq0,\ \overline{u}\neq0\ \mbox{(see (\ref{eq1}))}.
	\end{eqnarray*}
	
	So, $\overline{u}$ is a positive solution of (\ref{eq15}). As before, the nonlinear regularity theory and the nonlinear Hopf boundary point theorem (see \cite[p. 738]{10}) imply that $\overline{u}\in D_+$.
	
	Next, we show that there is a smallest positive solution for problem (\ref{eq15}). We first observe that from Papageorgiou, R\u adulescu \& Repov\v{s} \cite{24} (see the proof of Proposition 7), we know that the set $S_+$ of positive solutions of (\ref{eq15}) is downward directed (that is, if $\overline{u}_1,\, \overline{u}_2\in S_+$, then we can find $\overline{u}\in S_+$ such that $\overline{u}\leq\overline{u}_1$, $\overline{u}\leq\overline{u}_2$).
	Invoking Lemma 3.10 of Hu \& Papageorgiou \cite[p. 178]{14}, we can find a decreasing sequence $\{\overline{u}_n\}_{n\geq1}\subseteq S_+\subseteq D_+$  such that
	\begin{equation*}
		\inf S_+ = \inf_{n\geq1}\overline{u}_n.
	\end{equation*}
	
	We have
	\begin{equation}\label{eq20}
		\langle A(\overline{u}_n),h\rangle + \int_{\partial\Omega}\beta(z)\overline{u}^{p-1}_n hd\sigma = c_5\int_\Omega\overline{u}^{r-1}_n hdz\ \mbox{for all}\ h\in W^{1,p}(\Omega),\ n\in\NN.
	\end{equation}
	
	It follows from (\ref{eq20}) that $\{\overline{u}_n\}_{n\geq1}\subseteq W^{1,p}(\Omega)$ is bounded. So, we may assume that
	\begin{equation}\label{eq21}
		\overline{u}_n\xrightarrow{w}\overline{u}_*\ \mbox{in}\ W^{1,p}(\Omega), \overline{u}_n\rightarrow\overline{u}_*\ \mbox{in}\ L^r(\Omega)\ \mbox{and}\ L^p(\partial\Omega).
	\end{equation}
	
	Suppose that $\overline{u}_*\equiv0$. Let $\overline{y}_n=\frac{\overline{u}_n}{||\overline{u}_n||}$, $n\in\NN$. Then $||\overline{y}_n||=1$ for all $n\in\NN$ and so we may assume that
	\begin{equation}\label{eq22}
		\overline{y}_n\xrightarrow{w}\overline{y}\ \mbox{in}\ W^{1,p}(\Omega), \overline{y}_n\rightarrow\overline{y}\ \mbox{in}\ L^p(\Omega)\ \mbox{and}\ L^p(\partial\Omega).
	\end{equation}
	
	From (\ref{eq20}) we have
	\begin{equation*}
		\langle A(\overline{y}_n),h\rangle + \int_{\partial\Omega}\beta(z)\overline{y}^{p-1}_nhd\sigma = c_5\int_\Omega \overline{u}^{r-p}_n\ \overline{y}^{p-1}_n\ hdz\ \mbox{for all}\ h\in W^{1,p}(\Omega),\ n\in\NN.
	\end{equation*}
	
	Choosing $h=\overline{y}_n-\overline{y}\in W^{1,p}(\Omega)$, passing to the limit as $n\rightarrow\infty$, and using (\ref{eq22}) and the fact that $\overline{u}_*=0$, we obtain
	\begin{eqnarray}
		&& \lim_{n\rightarrow\infty} \langle A(\overline{y}_n), \overline{y}_n-y\rangle = 0, \nonumber \\
		&\Rightarrow & \overline{y}_n\rightarrow\overline{y}\ \mbox{in}\ W^{1,p}(\Omega),\ ||\overline{y}||=1.\ \mbox{(see Proposition \ref{prop4})}. \label{eq23}
	\end{eqnarray}
	
	Passing to the limit as $n\rightarrow\infty$ in (\ref{eq22}), and using (\ref{eq23}) and the fact that $\overline{u}_+=0$, we obtain
	\begin{eqnarray*}
		&& \langle A(\overline{y}),h\rangle + \int_{\partial\Omega}\beta(z)(z)\overline{y}^{p-1}hd\sigma=0\ \mbox{for all}\ h\in W^{1,p}(\Omega), \\
		&\Rightarrow & \tau(\overline{y}) = 0, \\
		&\Rightarrow & \overline{y}=0\ \mbox{(see (\ref{eq1})), contradicting (\ref{eq23})}.
	\end{eqnarray*}
	
	So, $\overline{u}_*\neq0$. In (\ref{eq20}) we choose $h=\overline{u}_n-\overline{u}_*\in W^{1,p}(\Omega)$, pass to the limit as $n\rightarrow\infty$, and use (\ref{eq21}) and Proposition \ref{prop4}. Then
	\begin{equation*}
		\overline{u}_n\rightarrow\overline{u}_*\ \mbox{in}\ W^{1,p}(\Omega).
	\end{equation*}
	
	Hence, in the limit as $n\rightarrow\infty$ in (\ref{eq20}), we obtain
	\begin{eqnarray*}
		&&\langle A(\overline{u}_*),h\rangle + \int_{\partial\Omega}\beta(z)\overline{u}^{p-1}_*hd\sigma = c_5\int_\Omega\overline{u}^{r-1}_*hdz\ \mbox{for all}\ h\in W^{1,p}(\Omega) \\
		&\Rightarrow & \overline{u}_*\in S_+\ \mbox{and}\ \overline{u}_*=\inf S_+.
	\end{eqnarray*}
	
	Now let $\tilde{u}_\lambda\in D_+$ be a solution of (\ref{eq7l}) (see Proposition \ref{prop7}). We consider the Carath\'eodory function $\gamma(z,x)$ defined by
	\begin{equation}\label{eq24}
		\gamma(z,x)=\left\{
			\begin{array}{ll}
				c_5 (x^+)^{r-1}\ &\mbox{if}\ x\leq\tilde{u}_\lambda(z) \\
				c_5\tilde{u}_\lambda(z)^{r-1}\ &\mbox{if}\ \tilde{u}_\lambda(z)<x.
			\end{array}
		\right.
	\end{equation}
	
	We set $\Gamma(z,x)=\int^x_0\gamma(z,s)ds$ and consider the $C^1$-functional $\hat\xi: W^{1,p}(\Omega)\rightarrow\RR$ defined by
	\begin{equation*}
		\hat\xi(u) = \frac{1}{p}\tau(u) - \int_\Omega\Gamma(z,u)dz\ \mbox{for all}\ u\in W^{1,p}(\Omega).
	\end{equation*}
	
	It follows by (\ref{eq1}) and (\ref{eq24}) that $\hat\xi(\cdot)$ is coercive. Also, it is sequentially weakly lower semicontinuous. So, we can find $\overline{u}\in W^{1,p}(\Omega)$ such that
	\begin{equation}\label{eq25}
		\hat\xi(\overline{u})=\inf\{\hat\xi(u): u\in W^{1,p}(\Omega)\}.
	\end{equation}
	
	As before, since $r>p$, we have $\hat\xi(\overline{u})<0=\hat\xi(0)$, hence $\overline{u}\neq0$. From (\ref{eq25}) we have
	\begin{equation}\label{eq26}
		\langle A(\overline{u}),h\rangle + \int_{\partial\Omega}\beta(z)|\overline{u}|^{p-2}\overline{u}hd\sigma = \int_\Omega\gamma(z,\overline{u})hdz\ \mbox{for all}\ h\in W^{1,p}(\Omega).
	\end{equation}
	
	In (\ref{eq26})  we first choose $h=-\overline{u}^-\in W^{1,p}(\Omega)$. We obtain
	\begin{eqnarray*}
		& \tau(\overline{u}^-)=0\ \mbox{(see (\ref{eq24}))}, \\
		\Rightarrow & \overline{u}\geq0,\ \overline{u}\neq0\ \mbox{(see (\ref{eq1}))}.
	\end{eqnarray*}
	
	Next, we choose $h=(\overline{u}-\tilde{u}_\lambda)^+\in W^{1,p}(\Omega)$ in (\ref{eq26}). Then
	\begin{eqnarray*}
		&& \langle A(\overline{u}),(\overline{u}-\tilde{u}_\lambda)^+\rangle + \int_{\partial\Omega}\beta(z)\overline{u}^{p-1}(\overline{u}-\tilde{u}_\lambda)^+d\sigma \\
		&= & \int_\Omega c_5\tilde{u}^{r-1}_\lambda (\overline{u}-\tilde{u}_\lambda)^+dz \\
		&\leq & \int_\Omega([\vartheta(z) + (1+\lambda)\varepsilon]\tilde{u}^{p-1}_\lambda + c_5\tilde{u}^{r-1}_\lambda)(\overline{u}-\tilde{u}_\lambda)^+dz \\
		&= & \langle A(\tilde{u}_\lambda), (\overline{u}-\tilde{u}_\lambda)^+\rangle + \int_{\partial\Omega}\beta(z)\tilde{u}^{p-1}_\lambda(\overline{u}-\tilde{u}_\lambda)^+d\sigma \\
		&\Rightarrow & \overline{u}\leq \tilde{u}_\lambda\ \mbox{(by Proposition \ref{prop4})}.
	\end{eqnarray*}
	
	So, we have proved that
	\begin{equation}
		\overline{u}\in [0,\tilde{u}_\lambda],\ \overline{u}\neq0.
		\label{eq26'}
	\end{equation}
	
	It follows by (\ref{eq24}), (\ref{eq26}) and \eqref{eq26'} that $\overline{u}\in S_+\subseteq D_+$. Therefore
	\begin{equation*}
		0<\hat{c}=\min_{\overline\Omega}\overline{u}_*\leq\tilde{u}_\lambda\ \mbox{for all}\ \lambda>0.
	\end{equation*}
	
	The oddness of (\ref{eq15}) implies that $\tilde{v}_\lambda\leq-\hat{c}<0$ for all $\lambda>0$.
\end{proof}

Now we are ready to produce two nontrivial constant sign solutions when $\lambda>0$ is large enough.
\begin{prop}\label{prop9}
	If hypotheses $H(\beta), H(g), H(f), H_0$ hold, then for sufficiently large $\lambda>0$ problem (\ref{eqp}) has two constant sign solutions
	\begin{equation*}
		u_0\in {\rm int}_{C^1(\overline\Omega)}[0,\tilde{u}_\lambda], v_0\in {\rm int}_{C^1(\overline\Omega)}[\tilde{v}_\lambda,0].
	\end{equation*}
	with $\tilde{u}_\lambda\in D_+$ and $\tilde{v}_\lambda\in -D_+$ constant sign solutions of (\ref{eq7l}).
\end{prop}
\begin{proof}
	We introduce the following truncation of the reaction in problem (\ref{eqp}):
	\begin{equation}\label{eq27}
		\eta^+_\lambda(z,x) = \left\{
			\begin{array}{ll}
				\lambda g(z,x^+) + f(z,x^+)\ &\mbox{if}\ x\leq\tilde{u}_\lambda(z) \\
				\lambda g(z,\tilde{u}_\lambda(z)) + f(z,\tilde{u}_\lambda(z))\ &\mbox{if}\ \tilde{u}_\lambda(z)<x.
			\end{array}
		\right.
	\end{equation}
	
	This is a Carath\'eodory function. We set $H^+_\lambda(z,x)=\int^x_0\eta^+_\lambda(z,s)ds$ and consider the $C^1$-functional $d^+_\lambda:W^{1,p}(\Omega)\rightarrow\RR$ defined by
	\begin{equation*}
		d^+_\lambda(u) = \frac{1}{p}\tau(u) - \int_\Omega H^+_\lambda(z,u)dz\ \mbox{for all}\ u\in W^{1,p}(\Omega).
	\end{equation*}
	
	From (\ref{eq1}) and (\ref{eq27}) we see that $d^+_\lambda(\cdot)$ is coercive. Also, it is sequentially weakly lower semicontinuous. So, we can find $u_0\in W^{1,p}(\Omega)$ such that
	\begin{equation}\label{eq28}
		d^+_\lambda(u_0) = \inf\{d^+_\lambda(u): u\in W^{1,p}(\Omega)\}.
	\end{equation}
	
	Let $c\in(0,\hat{c})$ with $\hat{c}>0$ as in Proposition \ref{prop8}. Then for all $\lambda>0$, we have
	\begin{equation*}
		d^+_\lambda(c) = \frac{c^p}{p}\int_{\partial\Omega}\beta(z)d\sigma - \lambda\int_\Omega G(z,c)cdz - \int_\Omega F(z,c)cdz.
	\end{equation*}
	
	Note that $\int_\Omega G(z,c)cdz>0$ (see hypothesis $H(g)(i)$). So,
	\begin{eqnarray*}
		&& d^+_\lambda(c)<0\ \mbox{for sufficiently large}\ \lambda>0, \\
		&\Rightarrow & d^+_\lambda(u_0)<0 = d^+_\lambda(0)\ \mbox{for sufficiently large}\ \lambda>0, \\
		&\Rightarrow & u_0\neq0.
	\end{eqnarray*}
	
	From (\ref{eq28}) we have
	\begin{equation}\label{eq29}
		\langle A(u_0),h\rangle + \int_{\partial\Omega}\beta(z)|u_0|^{p-2}u_0hd\sigma = \int_\Omega\eta^+_\lambda(z,u_0)hdz\ \mbox{for all}\ h\in W^{1,p}(\Omega).
	\end{equation}
	
	In (\ref{eq29}) we first choose $h=-u^-_0\in W^{1,p}(\Omega)$ and obtain
	\begin{equation*}
		u_0\geq0,\ u_0\neq0.
	\end{equation*}
	
	Then in (\ref{eq29}) we choose $h=(u_0-\tilde{u}_\lambda)^+\in W^{1,p}(\Omega)$. As in the proof of Proposition \ref{prop8}, using this time (\ref{eq6}), we obtain
	\begin{equation*}
		u_0\leq \tilde{u}_\lambda.
	\end{equation*}
	
	So, we have proved that
	\begin{equation}\label{eq30}
		u_0\in[0,\tilde{u}_\lambda],\ u_0\neq0.
	\end{equation}
	
	By (\ref{eq27}), (\ref{eq29}), (\ref{eq30}) and Theorem 2 of Lieberman \cite{17}, we have that
	\begin{equation*}
		u_0\in C_+\backslash\{0\}\ \mbox{is a positive solution of (\ref{eqp})},\ \lambda>0\ \mbox{large enough}.
	\end{equation*}
	
	Therefore we have
	\begin{eqnarray*}
		&& \Delta_p u_0(z)\leq0\ \mbox{for almost all}\ z\in\Omega, \\
		&\Rightarrow & u_0\in D_+\ \mbox{(see Gasinski \& Papageorgiou \cite[p. 738]{10})}.
	\end{eqnarray*}
	
	Also, we have
	\begin{eqnarray*}
		& -\Delta_p u_0(z) &= \lambda g(z,u_0(z)) + f(z,u_0(z)) \\
		& &\leq \lambda g(z,\tilde{u}_\lambda(z)) + f(z,\tilde{u}_\lambda(z))\ \mbox{(see (\ref{eq30}) and hypothesis $H_0$)} \\
		& &< [\vartheta(z) + (1+\lambda)\varepsilon]\tilde{u}_\lambda(z)^{p-1} + c_5\tilde{u}_\lambda (z)^{r-1}\ \mbox{(see (\ref{eq6}))} \\
		& &= -\Delta_p\tilde{u}_\lambda(z)\ \mbox{for almost all}\ z\in\Omega, \\
		\Rightarrow & \tilde{u}_\lambda-u_0\in &{\rm int}\, C_+\ \mbox{(see Proposition \ref{prop3})}.
	\end{eqnarray*}
	
	We conclude that
	\begin{equation*}
		u_0\in {\rm int}_{C^1(\overline\Omega)}[0,\tilde{u}_\lambda].
	\end{equation*}
	
	For the negative solution, we introduce the Carath\'eodory function $\eta^-_\lambda(z,x)$ defined by
	\begin{equation*}
		\eta^-_\lambda(z,x)=\left\{
			\begin{array}{ll}
				\lambda g(z,\tilde{v}_\lambda(z)) + f(z,\tilde{v}_\lambda(z))\ &\mbox{if}\ x<\tilde{v}_\lambda(z) \\
				\lambda g(z,-x^-) + f(z,-x^-)\ &\mbox{if}\ \tilde{v}_\lambda(z)\leq x.
			\end{array}
		\right.
	\end{equation*}
	
	We set $H^-_\lambda(z,x)=\int^x_0\eta^-_\lambda(z,s)ds$ and consider the $C^1$-functional $d^-_\lambda:W^{1,p}(\Omega)\rightarrow\RR$ defined by
	\begin{equation*}
		d^-_\lambda(u) = \frac{1}{p}\tau(u) - \int_\Omega H^-_\lambda(z,u)dz\ \mbox{for all}\ u\in W^{1,p}(\Omega).
	\end{equation*}
	
	Working as above, this time with the functional $d^-_\lambda(\cdot)$, we produce a solution $v_0$ of (\ref{eqp}) for large enough $\lambda>0$ such that
	\begin{equation*}
		v_0\in -D_+,\ v_0\in {\rm int}_{C^1(\overline\Omega)}[\tilde{v}_\lambda,0].
	\end{equation*}
The proof is now complete.
\end{proof}

Using $u_0\in D_+$ and $v_0\in-D_+$ from Proposition \ref{prop9}, we will produce two more constant sign solutions.

\begin{prop}\label{prop10}
	If hypotheses $H(\beta), H(g), H(f), H_0$ hold, then for large enough $\lambda>0$, problem (\ref{eqp}) admits two more constant sign solutions $\hat{u}\in D_+$ and $\hat{v}\in -D_+$ such that
	\begin{equation*}
		\hat{u}-u_0\in {\rm int}\, C_+,\ v_0-\hat{v}\in {\rm int}\,C_+.
	\end{equation*}
	with $u_0\in D_+$ and $v_0\in-D_+$ the solutions from Proposition \ref{prop9}.
\end{prop}
\begin{proof}
	We introduce the following truncation of the reaction in problem \eqref{eqp}:
	\begin{equation}\label{eq31}
		i^+_\lambda(z,x)=\left\{\begin{array}{ll}
			\lambda g(z,u_0(z))+f(z,u_0(z))&\mbox{if}\ x\leq u_0(z)\\
			\lambda g(z,x)+f(z,x)&\mbox{if}\ u_0(z)<x.
		\end{array}\right.
	\end{equation}
	
	This is a Carath\'eodory function. We set $I^+_\lambda(z,x)=\int^x_0 i^+_\lambda(z,s)ds$ and consider the $C^1$-functional $\chi^+_\lambda:W^{1,p}(\Omega)\rightarrow\RR$ defined by
	$$\chi^+_\lambda(u)=\frac{1}{p}\tau(u)-\int_{\Omega}I^+_\lambda(z,u)dz\ \mbox{for all}\ u\in W^{1,p}(\Omega).$$
	\begin{claim}\label{cl3.1}
		$\chi^+_\lambda(\cdot)$ satisfies the $C$-condition.
	\end{claim}
	
	We consider a sequence $\{u_n\}_{n\geq 1}\subseteq W^{1,p}(\Omega)$ such that
	\begin{eqnarray}
		&&|\chi^+_\lambda(u_n)|\leq M_1\ \mbox{for some}\ M_1>0\ \mbox{and all}\ n\in\NN,\label{eq32}\\
		&&(1+||u_n||)(\chi^+_\lambda)'(u_n)\rightarrow 0\ \mbox{in}\ W^{1,p}(\Omega)^*\ \mbox{as}\ n\rightarrow\infty\label{eq33}.
	\end{eqnarray}
	
	From (\ref{eq33}) we have
	\begin{eqnarray}\label{eq34}
		&&|\left\langle A(u_n),h\right\rangle+\int_{\partial\Omega}\beta(z)|u_n|^{p-2}u_nhd\sigma-\int_\Omega i^+_\lambda(z,u_n)hdz|\leq\frac{\epsilon_n||h||}{1+||u_n||}\\
		&&\mbox{for all}\ h\in W^{1,p}(\Omega),\ \mbox{with}\ \epsilon_n\rightarrow 0^+.\nonumber
	\end{eqnarray}
	
	In (\ref{eq34}) we choose $h=-u^-_n\in W^{1,p}(\Omega)$. Then
	\begin{eqnarray}\label{eq35}
		&&\tau(u^-_n)\leq c_{14}||u^-_n||\ \mbox{for some}\ c_{14}>0\ \mbox{and all}\ n\in\NN\nonumber\\
		&&(\mbox{see (\ref{eq31}) and hypotheses}\ H(g)(i),H(f)(i)),\nonumber\\
		&\Rightarrow&\{u^-_n\}_{n\geq 1}\subseteq W^{1,p}(\Omega)\ \mbox{is bounded}\ (\mbox{see (\ref{eq2})}).
	\end{eqnarray}
	
	Using (\ref{eq35}) in (\ref{eq34}), we obtain
	\begin{eqnarray}\label{eq36}
		&&\left|\left\langle A(u^+_n),h\right\rangle+\int_{\partial\Omega}\beta(z)(u^+_n)^{p-1}hd\sigma-\int_\Omega i^+_\lambda(z,u^+_n)dz\right|\leq c_{15}||h||\\
		&&\mbox{for some}\ c_{15}>0\ \mbox{and all}\ h\in W^{1,p}(\Omega)\ n\in\NN.\nonumber
	\end{eqnarray}
	
	We will show that $\{u^+_n\}_{n\geq 1}\subseteq W^{1,p}(\Omega)$ is bounded, too. Arguing by contradiction, suppose that $||u^+_n||\rightarrow\infty$ as $n\rightarrow\infty$. Let $y_n=\frac{u^+_n}{||u^+_n||}$, $ n\in\NN$. Then $||y_n||=1,y_n\geq 0$ for all $n\in\NN$ and so we may assume that
	\begin{equation}\label{eq37}
		y_n\stackrel{w}{\rightarrow}y\ \mbox{in}\ W^{1,p}(\Omega)\ \mbox{and}\ y_n\rightarrow y\ \mbox{in}\ L^p(\Omega)\ \mbox{and}\ L^p(\partial\Omega),\ y\geq 0.
	\end{equation}
	
	From (\ref{eq36}) we have
	\begin{eqnarray}\label{eq38}
		&&\left|\left\langle A(y_n),h\right\rangle+\int_{\partial\Omega}\beta(z)y^{p-1}_nhd\sigma-\int_\Omega\frac{i^+_\lambda(z,u^+_n)}{||u^+_n||^{p-1}}hdz\right|\leq\frac{c_{15}||h||}{||u^+_n||^{p-1}}\\
		&&\mbox{for all}\ n\in\NN,\ h\in W^{1,p}(\Omega).\nonumber
	\end{eqnarray}
	
	From (\ref{eq5}) and (\ref{eq31}), we see that
	\begin{equation}\label{eq39}
		\left\{\frac{i^+_\lambda(\cdot,u^+_n(\cdot))}{||u^+_n||^{p-1}}\right\}_{n\geq 1}\subseteq L^{p'}(\Omega)\ \mbox{is bounded}\ \left(\frac{1}{p}+\frac{1}{p'}=1\right).
	\end{equation}
	
	Passing to a subsequence if necessary, and using hypotheses $H(g)(ii),H(f)(ii)$, we have
	\begin{equation}\label{eq40}
		\frac{i^+_\lambda(\cdot,u^+_n(\cdot))}{||u^+_n||^{p-1}}\stackrel{w}{\rightarrow}\hat{\lambda}_my^{p-1}\ \mbox{in}\ L^{p'}(\Omega)
	\end{equation}
	(see Aizicovici, Papageorgiou \& Staicu \cite{1}, proof of Proposition \ref{prop16}).
	
	In (\ref{eq38}) we choose $h=y_n-y\in W^{1,p}(\Omega)$, pass to the limit as $n\rightarrow\infty$, and use (\ref{eq37}) and (\ref{eq39}). Then
	\begin{eqnarray}\label{eq41}
		&&\lim\limits_{n\rightarrow\infty}\left\langle A(y_n),y_n-y\right\rangle=0,\nonumber\\
		&\Rightarrow&y_n\rightarrow y\ \mbox{in}\ W^{1,p}(\Omega)\ (\mbox{see Proposition \ref{prop4}})\ \mbox{and so}\ ||y||=1,\ y\geq 0.
	\end{eqnarray}
	
	So, if in (\ref{eq38}) we pass to the limit as $n\rightarrow\infty$, and use (\ref{eq40}) and (\ref{eq41}) to obtain
	\begin{eqnarray}\label{eq42}
		&&\left\langle A(y),h\right\rangle+\int_{\partial\Omega}\beta(z)y^{p-1}hd\sigma=\hat{\lambda}_m\int_\Omega y^{p-1}hdz\ \mbox{for all}\ h\in W^{1,p}(\Omega),\nonumber\\
		&\Rightarrow&-\Delta_py(z)=\hat{\lambda}_my(z)^{p-1}\ \mbox{for a.a.}\ z\in\Omega,\ \frac{\partial y}{\partial n_p}+\beta(z)y^{p-1}=0\ \mbox{on}\ \partial\Omega
	\end{eqnarray}
	(see Papageorgiou \& R\u adulescu \cite{22}).
	
	Since $m\geq 2$, it follows by (\ref{eq42}) that $y(\cdot)$ must be nodal, a contradiction to (\ref{eq41}). Therefore
	\begin{eqnarray*}
		&&\{u^+_n\}_{n\geq 1}\subseteq W^{1,p}(\Omega)\ \mbox{is bounded},\\
		&\Rightarrow&\{u_n\}_{n\geq 1}\subseteq W^{1,p}(\Omega)\ \mbox{is bounded (see (\ref{eq35}))}.
	\end{eqnarray*}
	
	So, we may assume that
	\begin{equation}\label{eq43}
		u_n\stackrel{w}{\rightarrow}u\ \mbox{in}\ W^{1,p}(\Omega)\ \mbox{and}\ u_n\rightarrow u\ \mbox{in}\ L^p(\Omega)\ \mbox{and in}\ L^p(\partial\Omega).
	\end{equation}
	
	In (\ref{eq34}) we choose $h=u_n-u\in W^{1,p}(\Omega)$, pass to the limit as $n\rightarrow\infty$, and use (\ref{eq43}) and the fact that $\{i^+_\lambda(\cdot, u_n(\cdot))\}_{n\geq 1}\subseteq L^{p'}(\Omega)$ is bounded (see (\ref{eq5}) and (\ref{eq31})). We obtain
	\begin{eqnarray*}
		&&\lim\limits_{n\rightarrow\infty}\left\langle A(u_n),u_n-u\right\rangle=0,\\
		&\Rightarrow&u_n\rightarrow u\ \mbox{in}\ W^{1,p}(\Omega)\ (\mbox{see Proposition \ref{prop4}}).
	\end{eqnarray*}
	
	So $\chi^+_\lambda(\cdot)$ satisfies the $C$-condition. This proves Claim \ref{cl3.1}.
	\begin{claim}\label{cl3.2}
		We may assume that $u_0\in D_+$ is a local minimizer of $\chi^+_\lambda(\cdot)$.
	\end{claim}
	
	For sufficiently large $\lambda>0$, as in Proposition \ref{prop9}, let $\tilde{u}_\lambda\in D_+$ be a solution of \eqref{eq7l} (see Proposition \ref{prop7}). From Proposition \ref{prop9} we know that
	\begin{equation}\label{eq44}
		\tilde{u}_\lambda-u_0\in {\rm int}\, C_+.
	\end{equation}
	
	We introduce the following truncation of $i^+_\lambda(z,\cdot)$:
	\begin{equation}\label{eq45}
		j^+_\lambda(z,x)=\left\{\begin{array}{ll}
			i^+_\lambda(z,x)&\mbox{if}\ x\leq\tilde{u}_\lambda(z)\\
			i^+_\lambda(z,\tilde{u}_\lambda(z))&\mbox{if}\ \tilde{u}_\lambda(z)<x.
		\end{array}\right.
	\end{equation}
	
	This is a Carath\'eodory function. We set $J^+_\lambda(z,x)=\int^x_0 j^+_\lambda(z,s)ds$ and consider the $C^1$-functional $\hat{\chi}^+_\lambda:W^{1,p}(\Omega)\rightarrow\RR$ defined by
	$$\hat{\chi}^+_\lambda(u)=\frac{1}{p}\tau(u)-\int_{\Omega}J^+_\lambda(z,u)dz\ \mbox{for all}\ u\in W^{1,p}(\Omega).$$
	
	By (\ref{eq1}) and (\ref{eq45}), it is clear that $\hat{\chi}^+_\lambda(\cdot)$ is coercive. Also, it is sequentially weakly lower semicontinuous. Hence we can find $\hat{u}_0\in W^{1,p}(\Omega)$ such that
	\begin{equation}\label{eq46}
		\hat{\chi}^+_\lambda(\hat{u}_0)=\inf\{\hat{\chi}^+_\lambda(u):u\in W^{1,p}(\Omega)\}.
	\end{equation}
	
	Using (\ref{eq45}), the nonlinear regularity theory and the nonlinear maximum principle, we can easily show that
	\begin{equation}\label{eq47}
		K_{\hat{\chi}^+_\lambda}\subseteq[u_0,\tilde{u}_\lambda]\cap D_+.
	\end{equation}
	
	Evidently, $\tilde{u}_\lambda\notin K_{\hat{\chi}^+_\lambda}$ (see (\ref{eq6})). So, from (\ref{eq46}) and (\ref{eq47}), we have
	$$\hat{u}_0\in[u_0,\tilde{u}_\lambda]\cap D_+,\ \hat{u}_0\neq\tilde{u}_\lambda.$$
	
	If $\hat{u}_0\neq u_0$, then this is the desired second positive solution of \eqref{eqp} for sufficiently large $\lambda>0$, and using Proposition \ref{eq3}, we have
	$$\hat{u}_0-u_0\in {\rm int}\, C_+.$$
	
	Therefore we are done.
	
	So, we may assume that $\hat{u}_0=u_0\in D_+$. Note that
	\begin{equation}\label{eq48}
		\left.\chi^+_\lambda\right|_{[0,\tilde{u}_\lambda]}=\left.\hat{\chi}^+_\lambda\right|_{[0,\tilde{u}_\lambda]}\ (\mbox{see (\ref{eq45})}).
	\end{equation}
	
	From Proposition \ref{prop9}, we have
	\begin{equation}\label{eq49}
		u_0\in {\rm int}_{C^1(\overline{\Omega})}[0,\tilde{u}_\lambda].
	\end{equation}
	
	Then it follows from (\ref{eq46}), (\ref{eq47}), (\ref{eq48})  that
	\begin{eqnarray*}
		&&u_0\ \mbox{is a local}\ C^1(\overline{\Omega})-\mbox{minimizer of}\ \chi^+_\lambda(\cdot),\\
		&\Rightarrow&u_0\ \mbox{is a local}\ W^{1,p}(\Omega)-\mbox{minimizer of}\ \chi^+_\lambda(\cdot)\ (\mbox{see Proposition \ref{prop2}}).
	\end{eqnarray*}
	
	This proves Claim \ref{cl3.2}.
	
	Using (\ref{eq31}), we can show that
	$$K_{\chi^+_\lambda}\subseteq\left[u_0\right)\cap D_+.$$
	
	So, we may assume that $K_{\chi^+_\lambda}$ is finite, or otherwise we already have an infinity of positive solutions of \eqref{eqp} (for large enough $\lambda>0$) strictly bigger than $u_0$ and so we are done. Then on account of Claim \ref{cl3.2}, we can find sufficiently small $\rho\in(0,1)$ such that
	\begin{equation}\label{eq50}
		\chi^+_\lambda(u_0)<\inf\{\chi^+_\lambda(u):||u-u_0||=\rho\}=\tilde{m}^+_\lambda\ (\mbox{see \cite{1}}).
	\end{equation}
	
	From hypotheses $H(g)(ii),H(f)(ii)$ and since $m\geq 2$, we have
	\begin{equation}\label{eq51}
		\chi^+_\lambda(t\hat{u}_1)\rightarrow-\infty\ \mbox{as}\ t\rightarrow+\infty.
	\end{equation}
	
	Then (\ref{eq50}), (\ref{eq51}) and Claim \ref{cl3.1} permit the use of Theorem \ref{th1} (the mountain pass theorem). So, we can find $\hat{u}\in W^{1,p}(\Omega)$ such that
	\begin{equation}\label{eq52}
		\hat{u}\in K_{\chi^+_\lambda}\subseteq\left[u_0\right)\cap D_+\ \mbox{and}\ \tilde{m}^+_\lambda\leq\chi^+_\lambda(\hat{u}).
	\end{equation}
	
It follows from (\ref{eq50}), (\ref{eq52}) and (\ref{eq31}) that
	\begin{eqnarray*}
		&&\hat{u}\in D_+\ \mbox{is a second positive solution of \eqref{eqp} for sufficiently large}\ \lambda>0,\\
		&&u_0\leq\hat{u},\ u_0\neq\hat{u}.
	\end{eqnarray*}
	
	We have
	\begin{eqnarray}\label{eq53}
		-\Delta_pu_0(z)&=&\lambda g(z,u_0(z))+f(z,u_0(z))\nonumber\\
		&\leq&\lambda g(z,\hat{u}(z))+f(z,\hat{u}(z))\ (\mbox{see (\ref{eq52}) and hypothesis}\ H_0)\nonumber\\
		&=&-\Delta_p\hat{u}(z)\ \mbox{for almost all}\ z\in\Omega.
	\end{eqnarray}
	
	Note that $\lambda g(\cdot,u_0(\cdot))+f(\cdot,u_0(\cdot))\neq\lambda g(\cdot, \hat{u}(\cdot))+f(\cdot,\hat{u}(\cdot))$ (see hypothesis $H_0$). So, from (\ref{eq53}) and Proposition \ref{prop3}, we can infer that
	$$\hat{u}-u_0\in {\rm int}\, C_+.$$
	
	Similarly, for the second negative solution, we use $v_0\in-D_+$ from Proposition \ref{prop9}. So, we define
	$$i^-_\lambda(z,x)=\left\{\begin{array}{ll}
		\lambda g(z,x)+f(z,x)&\mbox{if}\ x\leq v_0(z)\\
		\lambda g(z,v_0(z))+f(z,v_0(z))&\mbox{if}\ v_0(z)<x.
	\end{array}\right.$$
	
	This is a Carath\'eodory function. We set $I^-_\lambda(z,x)=\int^x_0 i^-_\lambda(z,s)ds$ and consider the $C^1$-functional $\chi^-_\lambda: W^{1,p}(\Omega)\rightarrow\RR$ defined by
	$$\chi^-_\lambda(u)=\frac{1}{p}\tau(u)-\int_\Omega I^-_\lambda(z,u)dz\ \mbox{for all}\ u\in W^{1,p}(\Omega).$$
	Working as above, this time with $\chi^-_\lambda$ and truncating at $\tilde{v}_\lambda\in-D_+$ to produce $\hat{\chi}^-_\lambda(\cdot)$, we generate a second negative solution $\hat{v}$ of problem \eqref{eqp} for sufficiently large $\lambda>0$, such that
	$$\hat{v}\in-D_+\ \mbox{and}\ v_0-\hat{v}\in {\rm int}\, C_+.$$
This completes the proof.
\end{proof}

\section{The fifth solution}

So far we have four nontrivial smooth solutions, all with sign information (two positive and two negative). In this section, using the theory of critical groups, we establish the existence of a fifth nontrivial smooth solution distinct from the other four.
\begin{prop}\label{prop11}
	If hypotheses $H(\beta),H(g),H(f),H_0$ hold and $\lambda>0$ is sufficiently large, then $C_k(\chi^{\pm}_\lambda,\infty)$ $=0$ for all $k\in\NN_0$.
\end{prop}
\begin{proof}
	We present the proof for the functional $\chi^+_\lambda(\cdot)$, the proof for $\chi^-_\lambda(\cdot)$ being similar.
	
	Let $\lambda\in(\hat{\lambda}_m,\hat{\lambda}_{m+1})\backslash\hat{\sigma}(p)$ (recall that $\hat{\sigma}(p)$ denotes the set of eigenvalues of $-\Delta_p$ with Robin boundary condition) and consider the $C^1$-functional $\psi^+_\lambda:W^{1,p}(\Omega)\rightarrow\RR$ defined by
	$$\psi^+_\lambda(u)=\frac{1}{p}\tau(u)-\frac{\lambda}{p}||u^+||^p_p\ \mbox{for all}\ u\in W^{1,p}(\Omega).$$
	
	We consider the following homotopy
	$$h^+_\lambda(t,u)=(1-t)\chi^+_\lambda(u)+t\psi^+_\lambda(u)\ \mbox{for all}\ (t,u)\in[0,1]\times W^{1,p}(\Omega).$$
	\begin{claim}\label{cl4.1}
		We can find $\eta_0\in\RR$ and $\delta_0>0$ such that
		$$h^+_\lambda(t,u_0)\leq\eta_0\Rightarrow(1+||u||)||(h^+_\lambda)'_u(t,u)||_*\geq\delta_0\ \mbox{for all}\ t\in[0,1].$$
	\end{claim}
	
	To prove Claim \ref{cl4.1}, we argue indirectly. So, suppose Claim \ref{cl4.1} is not true. Evidently, $h^+_\lambda(\cdot,\cdot)$ maps bounded sets to bounded ones. Hence we can find $\{t_n\}_{n\geq 1}\subseteq[0,1]$ and $\{u_n\}_{n\geq 1}\subseteq W^{1,p}(\Omega)$ such that
	\begin{equation}\label{eq54}
		t_n\rightarrow t,\ ||u_n||\rightarrow\infty,\ h^+_\lambda(t_n,u_n)\rightarrow-\infty\ \mbox{and}\ (1+||u_n||)(h^+_\lambda)'_n(t_n,u_n)\rightarrow 0.
	\end{equation}
	
	From the last convergence in (\ref{eq54}), we have
	\begin{eqnarray}\label{eq55}
		\left|\left\langle A(u_n),h\right\rangle+
			\int_{\partial\Omega}\beta(z)|u_n|^{p-2}u_nhd\sigma-
		(1-t_n)\int_\Omega i^+_\lambda(z,u_n)hdz-\lambda t_n\int_\Omega(u^+_n)^{p-1}hdz\right|\nonumber\\
		\leq\frac{\epsilon_n||h||}{1+||u_n||}\ \mbox{for all}\ h\in W^{1,p}(\Omega),\ \mbox{with}\ \epsilon_n\rightarrow 0^+.
	\end{eqnarray}
	
	In (\ref{eq55}) we choose $h=-u^-_n\in W^{1,p}(\Omega)$. Then
	\begin{eqnarray}\label{eq56}
		&&\tau(u^-_n)\leq c_{16}||u^-_n||\ \mbox{for some}\ c_{16}>0\ \mbox{and all}\ n\in\NN\ (\mbox{see (\ref{eq31})}),\nonumber\\
		&\Rightarrow&\{u^-_n\}_{n\geq 1}\subseteq W^{1,p}(\Omega)\ \mbox{is bounded (see (\ref{eq1}))}.
	\end{eqnarray}
	
	From (\ref{eq54}) and (\ref{eq56}) it follows that
	$$||u^+_n||\rightarrow+\infty.$$
	
	We set $y_n=\frac{u^+_n}{||u^+_n||},\ n\in\NN$. Then $||y_n||=1,\ y_n\geq 0$ for all $n\in\NN$. So, we may assume that
	\begin{equation}\label{eq57}
		y_n\stackrel{w}{\rightarrow}y\ \mbox{in}\ W^{1,p}(\Omega)\ \mbox{and}\ y_n\rightarrow y\ \mbox{in}\ L^p(\Omega)\ \mbox{and in}\ L^p(\partial\Omega),\ y\geq 0.
	\end{equation}
	
	From (\ref{eq55}) and (\ref{eq56}) it follows that
	\begin{eqnarray}\label{eq58}
		\left|\left\langle A(y_n),h\right\rangle+\int_{\partial\Omega}\beta(z)y^{p-1}_nhd\sigma-(1-t_n)\int_\lambda\frac{i^+_\lambda(z,u^+_n)}{||u^+_n||^{p-1}}hdz
		- \lambda t_n\int_\lambda y^{p-1}_n hdz\right|\leq\epsilon'_n||h|| \nonumber     \\
		\mbox{for all}\ h\in W^{1,p}(\Omega),\ \mbox{with}\ \epsilon'_n\rightarrow 0^+
	\end{eqnarray}
	
	From (\ref{eq5}) and (\ref{eq31}), we see that
	\begin{equation}\label{eq59}
		\left\{\frac{i^+_\lambda(\cdot,u^+_n(\cdot))}{||u^+_n||^{p-1}}\right\}_{n\geq 1}\subseteq L^{p'}(\Omega)\ \mbox{is bounded}.
	\end{equation}
	
	Passing to a subsequence if necessary and using hypotheses $H(g)(ii)$ and $H(f)(ii)$ we have
	\begin{equation}\label{eq60}
		\frac{i^+_\lambda(\cdot,u^+_n(\cdot))}{||u^+_n||^{p-1}}\stackrel{w}{\rightarrow}\hat{\lambda}_my^{p-1}\ \mbox{in}\ L^{p'}(\Omega)\ (\mbox{see \cite{1}}).
	\end{equation}
	
	In (\ref{eq58}) we choose $h=y_n-y\in W^{1,p}(\Omega)$, pass to the limit as $n\rightarrow\infty$, and use (\ref{eq57}), (\ref{eq59}). Then
	\begin{eqnarray}\label{eq61}
		&&\lim\limits_{n\rightarrow\infty}\left\langle A(y_n),y_n-y\right\rangle=0,\nonumber\\
		&\Rightarrow&y_n\rightarrow y\ \mbox{in}\ W^{1,p}(\Omega)\ \mbox{(see Proposition \ref{prop4}), hence}\ ||y||=1,\ y\geq 0.
	\end{eqnarray}
	
	In (\ref{eq58}) we pass to the limit as $n\rightarrow\infty$, and use (\ref{eq60}), (\ref{eq61}) and the continuity of $A(\cdot)$ (Proposition \ref{prop4}). We obtain
	
	\begin{eqnarray}\label{eq62}
		\left\langle A(y),h\right\rangle+\int_{\partial\Omega}\beta(z)y^{p-1}hd\sigma  
		=[(1-t)\hat{\lambda}_m+t\lambda]\int_{\Omega}y^{p-1}hdz\ \mbox{for all}\ h\in W^{1,p}(\Omega)  \Rightarrow
		 \nonumber   \\
		-\Delta_py(z)=\lambda_t y(z)^{p-1}\ \mbox{for almost all}\ z\in\Omega,\ \frac{\partial y}{\partial n_p}+ 
		 \beta(z)y^{p-1}=0\ \mbox{on}\ \partial\Omega,
	\end{eqnarray}
	
	with $\lambda_t=(1-t)\hat{\lambda}_m+t\lambda$. We have
	\begin{equation}\label{eq63}
		\lambda_t\in[\hat{\lambda}_m,\hat{\lambda}_{m+1}).
	\end{equation}
	
	From (\ref{eq62}) and (\ref{eq63}) and since $m\geq 2$, we can infer that
	$$y=0\ \mbox{or}\ y\ \mbox{is nodal}.$$
	
	Both assertions contradict (\ref{eq61}). This proves Claim \ref{cl4.1}.
	
	Then Claim \ref{cl4.1} and Theorem 5.1.21 of Chang \cite[p. 334]{3} (see also Liang \& Su \cite[Proposition 3.2]{16}), imply that
	\begin{eqnarray}\label{eq64}
		&&C_k(h^+_\lambda(0,\cdot),\infty)=C_k(h^+_\lambda(1,\cdot),\infty)\ \mbox{for all}\ k\in\NN_0,\nonumber\\
		&\Rightarrow&C_k(\chi^+_\lambda,\infty)=C_k(\psi^+_\lambda,\infty)\ \mbox{for all}\ k\in\NN_0.
	\end{eqnarray}
	
	Now we consider the following homotopy
	$$\hat{h}^+_\lambda(t,u)=\psi^+_\lambda(u)-t\int_\Omega udz\ \mbox{for all}\ (t,u)\in[0,1]\times W^{1,p}(\Omega).$$
	\begin{claim}\label{cl4.2}
		$(\hat{h}^+_\lambda)'_u(t,u)\neq 0$ for all $t\in[0,1]$,  $u\in W^{1,p}(\Omega)\backslash\{0\}$.
	\end{claim}
	
	Again, we argue indirectly. So, suppose that for some $t\in[0,1]$ and $u\in W^{1,p}(\Omega)\backslash\{0\}$, we have
	\begin{eqnarray}\label{eq65}
		&&(\hat{h}^+_\lambda)'_u(t,u)=0,\nonumber\\
		&\Rightarrow&\left\langle A(u),h\right\rangle+\int_{\partial\Omega}\beta(z)|u|^{p-2}uhd\sigma=\lambda\int_\Omega(u^+)^{p-1}hdz+t\int_\Omega hdz\\
		&&\mbox{for all}\ h\in W^{1,p}(\Omega).\nonumber
	\end{eqnarray}
	
	In (\ref{eq65}) we choose $h=-u^-\in W^{1,p}(\Omega)$. Then
	\begin{eqnarray*}
		&&\tau(u^-)\leq 0,\\
		&\Rightarrow&u\geq 0,\ u\neq 0\ (\mbox{see (\ref{eq1})}).
	\end{eqnarray*}
	
	Hence (\ref{eq65}) becomes
	\begin{eqnarray}\label{eq66}
		\left\langle A(u),h\right\rangle+\int_{\partial\Omega}\beta(z)u^{p-1}hdz=\lambda\int_\Omega u^{p-1}hdz+t\int_\Omega hdz\ \mbox{for all}\ h\in W^{1,p}(\Omega)      \Rightarrow          \nonumber \\
		-\Delta_p u(z)=\lambda u(z)^{p-1}+t\ \mbox{for a.a.}\ z\in\Omega,\ \frac{\partial u}{\partial n_p}+\beta(z)u^{p-1}=0\ \mbox{on}\ \partial\Omega.
	\end{eqnarray}
	
	As before, the nonlinear regularity theory implies that $u\in C_+\backslash\{0\}$. Also, from (\ref{eq66}) we have
	\begin{eqnarray*}
		\Delta_pu(z)\leq 0\ \mbox{for almost all}\ z\in\Omega        \Rightarrow      \\
		u\in D_+\ (\mbox{see Gasinski \& Papageorgiou \cite[p. 738]{10}}).
	\end{eqnarray*}
	
	Let $v\in D_+$ and consider the function $R(v,u)(\cdot)$ from Section 2. Using Proposition \ref{prop6}, we get
	\begin{eqnarray*}
		0&\leq&\int_\Omega R(v,u)dz\\
		&=&||Dv||^p_p-\int_\Omega(-\Delta_pu)\frac{v^p}{u^{p-1}}dz+\int_{\partial\Omega}\beta(z)u^{p-1}\frac{v^p}{u^{p-1}}d\sigma\\
		&&(\mbox{via the nonlinear Green identity, see \cite[p. 211]{10}})\\
		&\leq&||Dv||^p_p-\lambda||v||^p_p+\int_{\partial\Omega}\beta(z)v^pd\sigma\ (\mbox{see (\ref{eq66})})\\
		&=&\tau(v)-\lambda||v||^p_p.
	\end{eqnarray*}
	
	Let $v=\hat{u}_1\in D_+$. Then
	$$0\leq[\hat{\lambda}_1-\lambda]<0\ (\mbox{since}\ \lambda>\hat{\lambda}_m,m\geq 2\ \mbox{and}\ ||\hat{u}_1||_p=1),$$
	a contradiction. This proves Claim \ref{cl4.2}.
	
	The homotopy invariance property of critical groups (see Gasinski \& Papageorgiou \cite[Theorem 5.125, p. 836]{11}) implies that for sufficiently small $r>0$ we have
	\begin{eqnarray}\label{eq67}
		&&H_k((\hat{h}^+_\lambda)(0,\cdot)^{\circ}\cap B_r,(\hat{h}^+_\lambda)(0,\cdot)^{\circ}\cap B_r\backslash\{0\})\nonumber\\
		&=&H_k((\hat{h}^+_\lambda)(1,\cdot)^{\circ}\cap B_r,(\hat{h}^+_\lambda)(1,\cdot)^{\circ}\cap B_r\backslash\{0\})\ \mbox{for all}\ k\in\NN_0.
	\end{eqnarray}
	
	On account of Claim \ref{cl4.2}, $0$ is an ordinary level for $\hat{h}^+_\lambda(1,\cdot)$. Hence from Granas \& Dugundji \cite[p. 407]{12}, we have
	\begin{equation}\label{eq68}
		H_k((\hat{h}^+_\lambda)(1,\cdot)^{\circ}\cap B_r,(\hat{h}^+_\lambda)(1,\cdot)^{\circ}\cap B_r\backslash\{0\})=0\ \mbox{for all}\ k\in\NN_0.
	\end{equation}
	
	From the definition of critical groups, we have
	\begin{eqnarray}\label{eq69}
		&&H_k((\hat{h}^+_\lambda)(0,\cdot)^{\circ}\cap B_r,(\hat{h}^+_\lambda)(0,\cdot)^{\circ}\cap B_r\backslash\{0\})=C_k(\psi^+_\lambda,0)\\
		&&\mbox{for all}\ k\in\NN_0.\nonumber
	\end{eqnarray}
	
	Combining (\ref{eq67}), (\ref{eq68}), (\ref{eq69}), we obtain
	\begin{equation}\label{eq70}
		C_k(\psi^+_\lambda,0)=0\ \mbox{for all}\ k\in\NN_0.
	\end{equation}
	
	Since $\lambda\in(\hat{\lambda}_m,\hat{\lambda}_{m+1})\backslash\hat{\sigma}(p)$, we have
	\begin{eqnarray}\label{eq71}
		&&K_{\psi^+_\lambda}=\{0\}\nonumber,\\
		&\Rightarrow&C_k(\psi^+_\lambda,0)=C_k(\psi^+_\lambda,\infty)\ \mbox{for all}\ k\in\NN_0\\
		&&\mbox{(see \cite[Proposition 6.61, p. 160]{19})}.\nonumber
	\end{eqnarray}
	
	By (\ref{eq64}), (\ref{eq70}), (\ref{eq71}), we can conclude that
	$$C_k(\chi^+_\lambda,\infty)=0\ \mbox{for all}\ k\in\NN_0.$$

	Similarly, we can show that
	$$C_k(\chi^-_\lambda,\infty)=0\ \mbox{for all}\ k\in\NN_0.$$
The proof is now complete.
\end{proof}

Let $\hat{u}\in D_+$ and $\hat{v}\in-D_+$ be the second pair of constant sign solutions for problem \eqref{eqp} ($\lambda>0$ sufficiently large) produced in Proposition \ref{prop10}.
\begin{prop}\label{prop12}
	If hypotheses $H(\beta),H(g),H(f),H_0$ hold and $\lambda>0$ is large enough (see Proposition \ref{prop10}), then $C_k(\chi^+_\lambda,\hat{u})=C_k(\chi^-_\lambda,\hat{v})=\delta_{k,1}\ZZ$ for all $k\in\NN_0$.
\end{prop}
\begin{proof}
	We may assume that $K_{\chi^+_\lambda}=\{u_0,\hat{u}\}$. Otherwise we already have a fifth nontrivial solution for \eqref{eqp}, which is also positive (see (\ref{eq52}) and (\ref{eq31})).
	
	Let $\hat{m}^+_\lambda=\chi^+_\lambda(u_0)$ and let $\tilde{m}^+_\lambda$ be as in (\ref{eq50}). We have $\hat{m}^+_\lambda<\tilde{m}^+_\lambda$ and we choose $\eta,\vartheta\in\RR$ such that
	\begin{equation}\label{eq72}
		\eta<\hat{m}^+_\lambda<\vartheta<\tilde{m}^+_\lambda.
	\end{equation}
	
	For these levels, we consider the corresponding sublevel sets for $\chi^+_\lambda$
	$$(\chi^+_\lambda)^{\eta}\subseteq(\chi^+_\lambda)^{\vartheta}\subseteq W^{1,p}(\Omega).$$
	
	For this triple we consider the corresponding long exact sequence of singular homological groups (see Motreanu, Motreanu \& Papageorgiou \cite[Proposition 6.14, p. 143]{19}). We have
	\begin{equation}\label{eq73}
		\cdots\rightarrow H_k(W^{1,p}(\Omega),(\chi^+_\lambda)^{\eta})\stackrel{i_*}{\rightarrow}H_k(W^{1,p}(\Omega),(\chi^+_\lambda)^\vartheta)
\stackrel{\hat{\partial}_*}{\rightarrow}H_{k-1}((\chi^+_\lambda)^\eta,(\chi^+_\lambda)^\vartheta)\rightarrow\cdots
	\end{equation}
	with $i_*$ being the homomorphism induced by the inclusion map $i:(W^{1,p}(\Omega),(\chi^+_\lambda)^\eta)\rightarrow(W^{1,p}(\Omega),(\chi^+_\lambda)^\vartheta)$ and $\hat{\partial}_*$ is the composite boundary homomorphism.
	
	From (\ref{eq72}) we see that $\eta<\inf\chi^+_\lambda(K_{\chi^+_\lambda})$ and so
	\begin{equation}\label{eq74}
		H_k(W^{1,p}(\Omega),(\chi^+_\lambda)^\eta)=C_k(\chi^+_\lambda,\infty)=0\ \mbox{for all}\ k\in\NN_0\ (\mbox{see Proposition \ref{prop11}}).
	\end{equation}
	
	Also, from (\ref{eq72}) and (\ref{eq52}), we have
	\begin{eqnarray}
		&&H_k(W^{1,p}(\Omega),(\chi^+_\lambda)^\vartheta)=C_k(\chi^+_\lambda,\hat{u})\ \mbox{for all}\ k\in\NN_0,\label{eq75}\\
		&&H_{k-1}((\chi^+_\lambda)^\vartheta,(\chi^+_\lambda)^\eta)=C_{k-1}(\chi^+_\lambda,u_0)=
\delta_{k-1,0}\ZZ=\delta_{k,1}\ZZ\ \mbox{for all}\ k\in\NN_0\label{eq76}
	\end{eqnarray}
	(see \cite[Lemma 6.55, p. 175]{19} and Claim 2 of Proposition \ref{prop10}).
	
	Returning to (\ref{eq73}) and using (\ref{eq74}), (\ref{eq75}), (\ref{eq76}), we see that only the tail (that is, $k=1$) of the long exact sequence is nontrivial. Moreover, by the rank theorem and the exactness of (\ref{eq73}), we have
	\begin{eqnarray}\label{eq77}
		&{\rm rank}\, H_1(W^{1,p}(\Omega),(\chi^+_\lambda)^\vartheta)&={\rm rank}\, {\rm ker}\, \hat{\partial}_*+{\rm rank}\, {\rm im}\, \hat{\partial}_*\nonumber\\
		&&={\rm rank}\, {\rm im}\, i_*+{\rm rank}\, {\rm im}\, \hat{\partial}_*,\nonumber\\
		\Rightarrow&{\rm rank}\, C_1(\chi^+_\lambda,\hat{u})\leq 1&(\mbox{see (\ref{eq74}), (\ref{eq75}), (\ref{eq76})}).
	\end{eqnarray}
	
	From the proof of Proposition \ref{prop10} we know that $\hat{u}\in K_{\chi^+_\lambda}$ is of the mountain pass type. Therefore
	\begin{equation}\label{eq78}
		C_1(\chi^+_\lambda,\hat{u})\neq 0
	\end{equation}
	(see \cite[Corollary 6.81, p. 168]{19}).
	
It follows	from (\ref{eq77}) and (\ref{eq78}) that
	$$C_k(\chi^+_\lambda,\hat{u})=\delta_{k,1}\ZZ\ \mbox{for all}\ k\in\NN_0$$
	
	Similarly, we can show that
	$$C_k(\chi^-_\lambda,\hat{v})=\delta_{k,1}\ZZ\ \mbox{for all}\ k\in\NN_0.$$
This completes the proof.
\end{proof}

Let $\varphi_\lambda:W^{1,p}(\Omega)\rightarrow\RR$ be the energy functional for problem \eqref{eqp} defined by
$$\varphi_\lambda(u)=\frac{1}{p}\tau(u)-\lambda\int_\Omega G(z,u)dz-\int_\Omega F(z,u)dz\ \mbox{for all}\ u\in W^{1,p}(\Omega).$$

Evidently, $\varphi_\lambda\in C^1(W^{1,p}(\Omega),\RR).$
\begin{prop}\label{prop13}
	If hypotheses $H(\beta),H(g),H(f)$ hold and $\lambda>0$, then the functional $\varphi_\lambda$ satisfies the $C$-condition.
\end{prop}
\begin{proof}
	Let $\{u_n\}_{n\geq 1}\subseteq W^{1,p}(\Omega)$ be a sequence such that
	\begin{eqnarray}
		&&|\varphi_\lambda(u_n)|\leq M_2\ \mbox{for some}\ M_2>0\ \mbox{and all}\ n\in\NN\label{eq79}\\
		&&(1+||u_n||)\varphi'_\lambda(u_n)\rightarrow 0\ \mbox{in}\ W^{1,p}(\Omega)^*\ \mbox{as}\ n\rightarrow\infty.\label{eq80}
	\end{eqnarray}
	
	By (\ref{eq80}) we have
	\begin{eqnarray}\label{eq81}
		&&\left|\left\langle A(u_n),h\right\rangle+\int_{\partial\Omega}\beta(z)|u_n|^{p-2}u_nhd\sigma-\lambda\int_\Omega g(z,u_n)hdz-\int_\Omega f(z,u_n)h dz\right|\nonumber\\
		&&\leq\frac{\epsilon_n||h||}{1+||u_n||}
	\end{eqnarray}
	for all $h\in W^{1,p}(\Omega)$ with $\epsilon_n\rightarrow 0^+$.
	
	In (\ref{eq81}) we choose $h=u_n\in W^{1,p}(\Omega)$. Then
	\begin{equation}\label{eq82}
		\tau(u_n)-\int_\Omega[\lambda g(z,u_n)+f(z,u_n)]u_ndz\leq\epsilon_n\ \mbox{for all}\ n\in\NN.
	\end{equation}
	
	Also, from (\ref{eq79}) we have
	\begin{equation}\label{eq83}
		-\tau(u_n)+\int_\Omega p[\lambda G(z,u_n)+F(z,u_n)]dz\leq pM_2\ \mbox{for all}\ n\in\NN.
	\end{equation}
	
	Adding (\ref{eq82}) and (\ref{eq83}), we obtain
	\begin{eqnarray}\label{eq84}
		&&\int_\Omega[p(\lambda G(z,u_n)+F(z,u_n))-(\lambda g(z,u_n)+f(z,u_n))u_n]dz\leq M_3\\
		&&\mbox{for some}\ M_3>0\ \mbox{and all}\  n\in\NN.\nonumber
	\end{eqnarray}
	
	We claim that $\{u_n\}_{n\geq 1}\subseteq W^{1,p}(\Omega)$ is bounded. Arguing by contradiction, suppose that $||u_n||\rightarrow\infty$. We set $y_n=\frac{u_n}{||u_n||},\ n\in\NN$. We have $||y_n||=1$ for all $n\in\NN$ and so we may assume that
	\begin{equation}\label{eq85}
		y_n\stackrel{w}{\rightarrow} y\ \mbox{in}\ W^{1,p}(\Omega)\ \mbox{and}\ y_n\rightarrow y\ \mbox{in}\ L^p(\Omega)\ \mbox{and}\ L^p(\partial\Omega).
	\end{equation}
	
	From (\ref{eq81}) we have
	\begin{eqnarray}\label{eq86}
		&&\left|\left\langle A(y_n),h\right\rangle+\int_{\partial\Omega}\beta(z)|y_n|^{p-2}y_nhd\sigma-\int_\Omega\frac{\lambda g(z,u_n)}{||u_n||^{p-1}}hdz-\int_\Omega\frac{f(z,u_n)}{||u_n||^{p-1}}hdz\right|\nonumber\\
		&&\leq\frac{\epsilon_n||h||}{(1+||u_n||)||u_n||^{p-1}}
	\end{eqnarray}
	for all $n\in\NN$.
	
	From (\ref{eq5}) it is clear that
	\begin{equation}\label{eq87}
		\left\{\frac{g(\cdot,u_n(\cdot))}{||u_n||^{p-1}}\right\}_{n\in\NN},\ \left\{\frac{f(\cdot,u_n(\cdot))}{||u_n||^{p-1}}\right\}_{n\in\NN}\subseteq L^{p'}(\Omega)\ \mbox{are bounded sequences.}
	\end{equation}
	
	In (\ref{eq86}) we choose $h=y_n-y\in W^{1,p}(\Omega)$ and pass to the limit as $n\rightarrow\infty$. Then using (\ref{eq85}) and (\ref{eq87}), we obtain
	\begin{eqnarray}\label{eq88}
		&&\lim\limits_{n\rightarrow\infty}\left\langle A(y_n),y_n-y\right\rangle=0,\nonumber\\
		&\Rightarrow&y_n\rightarrow y\ \mbox{in}\ W^{1,p}(\Omega)\ (\mbox{see Proposition \ref{prop4}})\ \mbox{and so}\ ||y||=1.
	\end{eqnarray}
	
	From (\ref{eq88}) we see that $y\neq 0$ and so if $D_+=\{z\in\Omega:|y(z)|>0\}$, then $|D_+|_N>0$ with $|\cdot|_N$ denoting the Lebesgue measure on $\RR^N$. We have
	\begin{eqnarray}\label{eq89}
		&&|u_n(z)|\rightarrow+\infty\ \mbox{for almost all}\ z\in D_+,\nonumber\\
		&\Rightarrow&\liminf\limits_{n\rightarrow\infty}\frac{pF(z,u_n(z))-f(z,u_n(z))u_n(z)}{|u_n(z)|^\tau}>0\ \mbox{for almost all}\ z\in D_+\nonumber\\
		&&\mbox{(see hypothesis H(f)(iii))}\nonumber\\
		&\Rightarrow&\liminf\limits_{n\rightarrow\infty}\frac{1}{||u_n||^\tau}\int_{D_+}[pF(z,u_n)-f(z,u_n)u_n]dz>0\ \mbox{(by Fatou's lemma)}\nonumber\\
		&\Rightarrow&\liminf\limits_{n\rightarrow\infty}\frac{1}{||u_n||^\tau}\int_\Omega[pF(z,u_n)-f(z,u_n)u_n]dz>0\ (\mbox{see (\ref{eq5})}).
	\end{eqnarray}
	
	Note that hypothesis $H(g)(ii)$ implies that given $\epsilon>0$, we can find $c_{17}=c_{17}(\epsilon)>0$ such that
	\begin{equation}\label{eq90}
		g(z,x)x\leq\epsilon|x|^\tau+c_{17}\ \mbox{for almost all}\ z\in\Omega\ \mbox{and all}\ x\in\RR\ (\mbox{see hypothesis $H(g)(i)$}).
	\end{equation}
	
	Since $G(z,x)\geq 0$ for almost all $z\in\Omega$, all $x\in\RR$ (by the sign condition in $H(g)(i)$), we obtain
	\begin{equation}\label{eq91}
		pG(z,x)-g(z,x)x\geq-\epsilon|x|^\tau-c_{17}\ \mbox{for almost all}\ z\in\Omega\ \mbox{and all}\ x\in\RR\ (\mbox{see (\ref{eq90})}).
	\end{equation}
	
	Hence
	\begin{eqnarray*}
		&&\int_\Omega[p(\lambda G(z,u_n)+F(z,u_n))-(\lambda g(z,u_n)+f(z,u_n))u_n]dz\\
		&\geq&\int_\Omega[-\lambda\epsilon|u_n|^\tau+(pF(z,u_n)-f(z,u_n)u_n)]dz\ (\mbox{see (\ref{eq91})})\\
		\Rightarrow&&\frac{1}{||u_n||^\tau}\int_\Omega[p(\lambda G(z,u_n)+F(z,u_n))-(\lambda g(z,u_n)+f(z,u_n))u_n]dz\\
		&\geq&\int_\Omega\frac{-\lambda\epsilon}{||u_n||^{p-\tau}}|y_n|^\tau dz+\frac{1}{||u_n||^p}\int_\Omega[pF(z,u_n)-f(z,u_n)u_n]dz.
	\end{eqnarray*}
	
	Using (\ref{eq89}), we obtain
	\begin{equation}\label{eq92}
		\liminf\limits_{n\rightarrow\infty}\frac{1}{||u_n||^\tau}\int_\Omega[p(\lambda G(z,u_n)+F(z,u_n))-(\lambda g(z,u_n)+f(z,u_n))u_n]dz>0.
	\end{equation}
	
	On the other hand, relation (\ref{eq84}) yields
	\begin{equation}\label{eq93}
		\limsup\limits_{n\rightarrow\infty}\frac{1}{||u_n||^\tau}\int_\Omega[p(\lambda G(z,u_n)+F(z,u_n))-(\lambda g(z,u_n)+f(z,u_n))u_n]dz\leq 0.
	\end{equation}
	
	Comparing (\ref{eq92}) and (\ref{eq93}), we arrive at a contradiction.
	
	This proves that
	$$\{u_n\}_{n\geq 1}\subseteq W^{1,p}(\Omega)\ \mbox{is bounded}.$$
	
	So, we may assume that
	\begin{equation}\label{eq94}
		u_n\stackrel{w}{\rightarrow} u\ \mbox{in}\ W^{1,p}(\Omega)\ \mbox{and}\ u_n\rightarrow u\ \mbox{in}\ L^p(\Omega)\ \mbox{and}\ L^p(\partial\Omega).
	\end{equation}
	
	In (\ref{eq81}) we choose $h=u_n-u\in W^{1,p}(\Omega)$, pass to the limit as $n\rightarrow\infty$, and use (\ref{eq94}). Then
	\begin{eqnarray*}
		&&\lim\limits_{n\rightarrow\infty}\left\langle A(u_n),u_n-u\right\rangle=0,\\
		&\Rightarrow&u_n\rightarrow u\ \mbox{in}\ W^{1,p}(\Omega)\ (\mbox{see Proposition \ref{prop4}}),\\
		&\Rightarrow&\varphi_\lambda\ \mbox{satisfies the $C$-condition}.
	\end{eqnarray*}
The proof is now complete.
\end{proof}

Then using Proposition 8 of Papageorgiou, R\u adulescu \& Repov\v{s} \cite{25}, (see also \cite{4}), we obtain the following property.
\begin{prop}\label{prop14}
	If hypotheses $H(\beta),H(g),H(f)$ hold and $\lambda>0$, then $C_m(\varphi_\lambda,\infty)\neq 0$.
\end{prop}

We assume that $K_{\varphi_\lambda}$ ($\lambda>0$ sufficiently large, as in Proposition \ref{prop10}) is finite. Otherwise we already have an infinity of solutions which are in $C^1(\overline{\Omega})$ (nonlinear regularity theory).
\begin{prop}\label{prop15}
	If hypotheses $H(\beta),H(g),H(f),H_0$ hold and $\lambda>0$ is sufficiently large (see Proposition \ref{prop10}), then
	\begin{eqnarray*}
		&&C_k(\varphi_\lambda,\hat{u})=C_k(\varphi_\lambda,\hat{v})=\delta_{k,1}\ZZ\ \mbox{for all}\ k\in\NN_0,\\
		&&C_k(\varphi_\lambda,0)=C_k(\varphi_\lambda,u_0)=C_k(\varphi_\lambda,v_0)=\delta_{k,0}\ZZ\ \mbox{for all}\ k\in\NN_0.
	\end{eqnarray*}
\end{prop}
\begin{proof}
	Consider the homotopy $\tilde{h}^+_\lambda(\cdot,\cdot)$ defined by
	$$\tilde{h}^+_\lambda(t,u)=(1-t)\varphi_\lambda(u)+t\chi^+_\lambda(u)\ \mbox{for all}\ (t,u)\in[0,1]\times W^{1,p}(\Omega).$$
	
	Suppose that we can find $t_n\rightarrow t$ and $u_n\rightarrow \hat{u}$ in $W^{1,p}(\Omega)$ such that
	$$(\tilde{h}^+_\lambda)'_n(t_n,u_n)=0\ \mbox{for all}\ n\in\NN.$$
	
	We have
	\begin{eqnarray}\label{eq95}
		&&\left\langle A(u_n),h\right\rangle+\int_{\partial\Omega}\beta(z)|u_n|^{p-2}u_nhd\sigma\nonumber\\
		&=&(1-t_n)\int_\Omega(\lambda g(z,u_n)+f(z,u_n))hdz+\int_\Omega i^+_\lambda(z,u_n)hdz\nonumber\\
		&&\mbox{for all}\ h\in W^{1,p}(\Omega),\ n\in\NN,\nonumber\\
		\Rightarrow&&-\Delta_pu_n(z)=(1-t_n)(\lambda g(z,u_n(z))+f(z,u)n(z))+t_ni^+_\lambda(z,u_n(z)))\nonumber\\
		&&\mbox{for almost all}\ z\in\Omega,\nonumber\\
		&&\frac{\partial u}{\partial n_p}+\beta(z)|u_n|^{p-2}u_n=0\ \mbox{on}\ \partial\Omega\ (\mbox{see Papageorgiou \& R\u adulescu \cite{22}}).
	\end{eqnarray}
	
	From (\ref{eq95}) and Proposition 7 of Papageorgiou \& R\u adulescu \cite{23}, we have
	$$||u_n||_\infty\leq M_4\ \mbox{for some}\ M_4>0\ \mbox{and all}\ n\in\NN.$$
	
	Then invoking Theorem 2 of Lieberman \cite{17}, we can find $\alpha\in(0,1)$ and $M_5>0$ such that
	\begin{equation}\label{eq96}
		u_n\in C^{1,\alpha}(\overline{\Omega})\ \mbox{and}\ ||u_n||_{C^{1,\alpha}(\overline{\Omega})}\leq M_5\ \mbox{for some}\ M_5>0\ \mbox{and all}\ n\in\NN.
	\end{equation}
	
	By (\ref{eq96}), the compact embedding of $C^{1,\alpha}(\overline{\Omega})$ into $C^1(\overline{\Omega})$ and the fact that $u_n\rightarrow \hat{u}$ in $W^{1,p}(\Omega)$, we infer that
	\begin{eqnarray*}
		&&u_n\rightarrow\hat{u}\ \mbox{in}\ C^1(\overline{\Omega}),\\
		&\Rightarrow&u_n-u_0\in {\rm int}\, C_+\ \mbox{for all}\ n\geq n_0\ (\mbox{see Proposition \ref{prop10}}),\\
		&\Rightarrow&\{u_n\}_{n\geq n_0}\subseteq K_{\varphi_\lambda}\ (\mbox{see (\ref{eq31})}),
	\end{eqnarray*}
	a contradiction to our hypothesis that $K_{\varphi_\lambda}$ is finite.
	
	Therefore by the homotopy invariance property of critical groups (see \cite[p. 836]{11}), we have
	\begin{eqnarray*}
		&&C_k(\varphi_\lambda,\hat{u})=C_k(\chi^+_\lambda,\hat{u})\ \mbox{for all}\ k\in\NN_0,\\
		&\Rightarrow&C_k(\varphi_\lambda,\hat{u})=\delta_{k,1}\ZZ\ \mbox{for all}\ k\in\NN_0.
	\end{eqnarray*}
	
	Similarly, using this time $\chi^-_\lambda$, we show that
	$$C_k(\varphi_\lambda,\hat{v})=\delta_{k,1}\ZZ\ \mbox{for all}\ k\in\NN_0.$$
	
	Recall that $u_0\in D_+$ and $v_0\in-D_+$ are local minimizers of the functionals $\chi^+_\lambda(\cdot)$ and $\chi^-_\lambda(\cdot)$, respectively (see Claim \ref{cl4.2} in the proof of Proposition \ref{prop10}). Hence we have
	\begin{equation}\label{eq97}
		C_k(\chi^+_\lambda,u_0)=C_k(\chi^-_\lambda,v_0)=\delta_{k,0}\ZZ\ \mbox{for all}\ k\in\NN_0.
	\end{equation}
	
	A homotopy invariance argument as above, shows that
	\begin{eqnarray*}
		&&C_k(\varphi_\lambda,u_0)=C_k(\chi^+_\lambda,u_0)\ \mbox{and}\ C_k(\varphi_\lambda,v_0)=C_k(\chi^-_\lambda,v_0)\ \mbox{for all}\ k\in\NN_0,\\
		&\Rightarrow&C_k(\varphi_\lambda,u_0)=C_k(\varphi_\lambda,v_0)=\delta_{k,0}\ZZ\ \mbox{for all}\ k\in\NN_0\ (\mbox{see (\ref{eq97})}).
	\end{eqnarray*}
	
	Finally, hypotheses $H(g)(ii)$ and $H(f)(iii)$ imply that
	$$u=0\ \mbox{is a local minimizer of}\ \varphi_\lambda$$
	(see also the proof of Proposition \ref{prop7}). It follows that
	$$C_k(\varphi_\lambda,0)=\delta_{k,0}\ZZ\ \mbox{for all}\ k\in\NN_0.$$
The proof is now complete.
\end{proof}

\begin{prop}\label{prop16}
	If hypotheses $H(\beta),H(g),H(f),H_0$ hold and $\lambda>0$ is large enough  (see Proposition \ref{prop10}), then problem \eqref{eqp} has a fifth nontrivial solution
	$$y_0\in C^1(\overline{\Omega}).$$
\end{prop}
\begin{proof}
	According to Proposition \ref{prop14}, we have
	$$C_m(\varphi_\lambda,\infty)\neq 0.$$
	
	So, there exists $y_0\in K_{\varphi_\lambda}$ such that
	\begin{equation}\label{eq98}
		C_m(\varphi_\lambda,y_0)\neq 0.
	\end{equation}
	
	Since $m\geq 2$, by Proposition \ref{prop15} and (\ref{eq98}), we infer that
	$$y_0\notin\{0,u_0,v_0,\hat{u},\hat{v}\}.$$
	
	Therefore $y_0$ is a fifth nontrivial solution of \eqref{eqp} (for sufficiently large $\lambda>0$) and the nonlinear regularity theory implies that $y_0\in C^1(\overline{\Omega})$.
\end{proof}

Finally, we can state the following multiplicity theorem for problem \eqref{eqp}.
\begin{theorem}\label{th17}
	If hypotheses $H(\beta),H(g),H(f),H_0$ hold, then for all sufficiently large $\lambda>0$ problem \eqref{eqp} has at least five nontrivial solutions
	\begin{eqnarray*}
		&&u_0,\hat{u}\in D_+\ \mbox{with}\ \hat{u}-u_0\in {\rm int}\,C_+,\\
		&&v_0,\hat{v}\in-D_+\ \mbox{with}\ v_0-\hat{v}\in {\rm int}\, C_+\\
		&\mbox{and}&y_0\in C^1(\overline{\Omega}).
	\end{eqnarray*}
\end{theorem}

{\bf Question.}
	Is it possible, in the framework of the present paper, to generate nodal solutions for \eqref{eqp}?

\medskip
{\bf Acknowledgments.} This research was supported by the Slovenian Research Agency grants
P1-0292, J1-8131, J1-7025,  N1-0064, and N1-0083.

\end{document}